\documentclass[12pt,epsfig,amsfonts]{amsart}
\usepackage{amsmath,amsthm,amssymb,amscd,epsfig,color,bbm}
\usepackage{graphicx}
\usepackage{url}
\usepackage{mathrsfs}
\usepackage{ulem}
\usepackage{pb-diagram} 
\usepackage[all]{xy}

\newtheorem{prop}{Proposition}[section]
\newtheorem{lemma}[prop]{Lemma}

\newtheorem{thm}[prop]{Theorem}

\theoremstyle{definition}

\theoremstyle{remark}
\newtheorem{remark}[prop]{Remark}

\numberwithin{equation}{section}

\begin{document}

\author{Hiroki Takahasi}

 \address{Keio Institute of Pure and Applied Sciences (KiPAS), Department of Mathematics,
 Keio University, Yokohama,
 223-8522, JAPAN} 
 \email{hiroki@math.keio.ac.jp}


\date{}

\title[Large deviation principles 
for periodic points of the Dyck shift ]
 {Large deviation principles 
for\\ periodic points of the Dyck shift} 

 \maketitle

 \begin{abstract}
 We investigate periodic points of the Dyck shift from the viewpoint of large deviations.
We establish 
 the level-2 Large Deviation Principle
 with the rate function given in terms of
 Kolmogorov-Sinai entropies of shift-invariant Borel probability measures. 
 Unlike 
 topologically mixing Markov shifts, the level-2 rate function is
  non-convex and level-1 rate functions are superpositions of two convex continuous functions. 
 Using the 
thermodynamic formalism, we show the analyticity of level-1 rate functions in some relevant cases.
We 
display a non-convex level-1 rate function with a non-differentiable point in the interior of its effective domain.
      \end{abstract}

\section{Introduction}
 The theory of large deviations aims to characterize limit behaviors of probability measures in terms of rate functions. 
We say the {\it Large Deviation Principle} (LDP) holds for a sequence 
 $(\mu_n)_{n=1}^\infty$ of Borel probability measures on a 
 topological space $\mathcal X$
  if there exists a lower semicontinuous function $I\colon\mathcal X\to[0,\infty]$ such that
\[\tag{lower bound}
 \liminf_{n\to\infty}\frac{1}{n}\log \mu_n(\mathcal G)\geq -\inf_{\mathcal G} I\ 
\text{ for any open set } 
\mathcal G\subset\mathcal X\]
and
\[\tag{upper bound}\limsup_{n\to\infty}\frac{1}{n}\log \mu_n(\mathcal C)\leq-\inf_{\mathcal C}I \ 
\text{ for any closed set }\mathcal C\subset \mathcal X.\]
We say $x\in \mathcal X$ is a {\it minimizer} if $I(x)=0$ holds. 
 The LDP roughly means that in the limit $n\to\infty$ the measure $\mu_n$ assigns
      all but exponentially small mass 
      to the set $\{x\in\mathcal X\colon I(x)=0\}$
      of minimizers.
      The function $I$ is called a {\it rate function}.
If $\mathcal X$ is a metric space and $(\mu_n)_{n=1}^\infty$ satisfies the LDP, the rate function is unique. 
The domain $D(I)=\{x\in\mathcal X\colon I(x)<\infty\}$ is called the {\it effective domain}.
We say the rate function $I$ is {\it good} if $\{x\in\mathcal X\colon I(x)\leq c\}$ is compact for any $c>0$.

Let us introduce a standard setup for large deviations in dynamical systems, i.e., iterated maps.
Let $X$ be a topological space and
let
 $T\colon X\to X$ be a Borel map.
Let $\mathcal M(X)$ denote the space of Borel probability measures on $X$ endowed with the weak* topology (see $\S$\ref{weak*}),
and let $\mathcal M(X,T)$ denote the subspace of $\mathcal M(X)$ that consists of $T$-invariant elements. For each $\mu\in\mathcal M(X,T)$, 
let $h(\mu)$
denote the Kolmogorov-Sinai entropy of $\mu$ with respect to $T$.
For $x\in X$ and $n\in\mathbb N$ define
an {\it empirical measure}
\[V_{n}(T,x)=\frac{1}{n}(\delta_x+\delta_{T( x)}+\cdots+\delta_{T^{n-1}(x)})\in \mathcal M(X),\] 
where $\delta_y\in\mathcal M(X)$ denotes the unit point mass at $y\in X$.
The LDP for a sequence of Borel probability measures on $\mathcal M(X)$ associated with 
empirical measures is referred to as {\it level-2}. 
Given a Borel observable $f\colon X\to\mathbb R$, for $x\in X$ and $n\in\mathbb N$ define a Birkhoff sum \[S_nf(x)=f(x)+f(T(x))+\cdots+f(T^{n-1}(x)).\] 
LDPs for sequences of Borel probability measures on $\mathbb R$ associated with the Birkhoff average $(1/n)S_nf(x)$ are referred to as {\it level-1} (see \cite[Chapter~I]{Ell85}).

This paper is concerned with the LDP in the context of symbolic dynamics.
Let $S$ be a non-empty finite discrete set, called a {\it finite alphabet}.
Let $S^{\mathbb Z}$ denote 
the two-sided Cartesian product topological space of $S$.
The left shift acts continuously on $S^{\mathbb Z}$.
   A shift-invariant closed subset of $S^{\mathbb Z}$ is called a {\it subshift} over $S$, and $S^{\mathbb Z}$ is called the {\it full shift}.
Subshifts determined by transition matrices are called {\it Markov shifts}. 
For topologically mixing Markov shifts, the
level-2 LDP has already been 
well-established: see \cite{Kif90,OrePel89,Tak84} for distributions of empirical measures; see
 \cite{Kif94} for periodic points.


\subsection{The Dyck shift}

We investigate the
 level-2 LDP for periodic points of the Dyck shift, a subshift originating in the theory of languages \cite{AU68}. 
Let $M\geq2$ be an integer.
 Consider two different
alphabets consisting of $M$ symbols
\[D_\alpha=\{\alpha_1,\ldots,\alpha_M\}\ \text{ and }\ D_\beta=\{\beta_1,\ldots,\beta_M\},\] and set
 $D=D_\alpha\cup D_\beta.$
Let $D^*$ denote the set of finite words in $D$.
Consider the monoid with zero, with $2M$ generators in $D$ with relations 
\[\alpha_i\cdot\beta_j=\delta_{ij},\
0\cdot 0=0\ \text{ for } i,j\in\{1,\ldots,M\},\] \[\gamma\cdot 1= 1\cdot\gamma=\gamma,\
 \gamma\cdot 0=0\cdot\gamma=0\ 
\text{ for }\gamma\in D^*\cup\{ 1\},\]
where $\delta_{ij}$ denotes Kronecker's delta.
Define a map ${\rm red}\colon D^*\to D^*\cup\{0,1\}$ by
\[{\rm red}(\gamma_1\cdots\gamma_n)=\prod_{i=1}^n\gamma_i.\]
The subshift
\[
\Sigma_{D}=\{x=(x_i)_{i\in \mathbb Z}\in D^{\mathbb Z}\colon {\rm red}(x_j\cdots x_k)\neq0\ \text{ for all }j,k\in\mathbb Z\text{ with }j<k\}.\]
 is called the {\it Dyck shift}. If we interpret $D$ 
as a collection of $M$ brackets, $\alpha_i$ left and $\beta_i$ right in pair, then  
$\Sigma_D$ is the subshift whose admissible words are legally aligned brackets. The Dyck shift is a representative of property A subshifts \cite{Kri00}. 

Dynamical properties of the Dyck shift are not well-understood yet. 
As a counterexample to the conjecture of Weiss \cite{W70}, 
Krieger \cite{Kri74} proved that there exist exactly two ergodic measures of maximal entropy $\log(M+1)$, which are fully supported and Bernoulli.
Meyerovitch \cite{Mey08} identified tail invariant measures for both one- and two-sided Dyck shifts. For each of the two ergodic maximal entropy measures, exponential decay of correlations for H\"older continuous observables was proved in \cite{T23}. 
For periodic points of the Dyck shift,
Keller \cite{Kel91} investigated dynamical zeta functions. Hamachi and Inoue \cite{HI05} provided necessary and sufficient conditions for the existence of a proper embedding of an irreducible subshift of finite type into the Dyck shift in terms of periodic points and topological entropy.

\subsection{Level-2 LDPs}
Let $\sigma$ denote the left shift acting on $\Sigma_D$.
 For each $n\in\mathbb N$ let
\[{\rm Per}_n(\sigma)=\{x\in\Sigma_D\colon\sigma^nx=x\},\]
and define a Borel probability measure $\widetilde\mu_{n}$  on $\mathcal M(\Sigma_D)$ by 
\[\widetilde\mu_{n}(\cdot)=\frac{\#\{x\in {\rm Per}_{n}(\sigma )\colon V_{n}(\sigma,x)\in \cdot \}}{\#{\rm Per}_{n}(\sigma )  }.\]
Our main result is stated as follows.
 \begin{thm}[the level-2 LDP]\label{break-Dyck}
The LDP holds for $(\widetilde\mu_n)_{n=1}^\infty$.
The rate function $I\colon\mathcal M(\Sigma_D)\to[0,\infty]$ is not convex. The minimizers of $I$ are the two ergodic maximal entropy measures.
\end{thm}



A major difficulty in the ergodic theory of non-Markov subshifts is the lack of the specification property.
 This property, introduced by Bowen \cite{Bow71}, allows us to glue together a given collection of orbit segments appropriately to form one orbit. The specification property is useful in the construction of 
 invariant measures with prescribed properties. 
 For more details, see \cite{You90}.
 
 The specification property is rarely expected for non-Markov shifts: Schmeling \cite{Sch97} proved that almost every $\beta$-shift
does not have the specification property.
A breakthrough was made by Pfister and Sullivan \cite{PS05}, who introduced a certain weak form of specification that holds for any $\beta$-shift, and with this property established the level-2 LDP with a convex rate function for distributions of empirical measures with respect to the unique maximal entropy measure of any $\beta$-shift. 
The non-convexity of the rate function in Theorem~\ref{break-Dyck} is 
a manifestation of a severe lack of specification for the Dyck shift.
For topologically mixing Markov shifts, the level-2 rate function for periodic points is convex \cite{Kif94}. 

To provide a formula for the rate function $I$ in Theorem~\ref{break-Dyck}, 
for each $i\in\mathbb Z$ define a function $H_i\colon \Sigma_D\to\mathbb Z$ by
      \[H_i(x)=\begin{cases}\sum_{j=0}^{i-1} \sum_{k=1}^{M}(\delta_{{\alpha_k},x_j}-\delta_{\beta_k,x_j})&\text{ for }  i\geq1,\\\sum_{j=i}^{-1} \sum_{k=1}^{M}(\delta_{{\beta_k},x_j}-\delta_{\alpha_k,x_j})&\text{ for } i\leq -1,\\    0&\text{ for }i=0.\end{cases}\]
For $i$, $j\in\mathbb Z$ distinct,  let \[\{H_i=H_j\}=\{x\in\Sigma_D\colon H_i(x)=H_j(x)\},\] and define
    \[\label{3-sets}\begin{split}
A_\alpha&=\left\{x\in\Sigma_D\colon
\lim_{i\to\infty}H_i(x)
=\infty\ \text{ and } \ \lim_{i\to-\infty}H_i(x)=-\infty\right\};\\
A_\beta&=\left\{x\in\Sigma_D\colon
\lim_{i\to\infty}H_i(x)=-\infty\ \text{ and } \
\lim_{i\to-\infty}H_i(x)=\infty\right\};\\
A_0&=\bigcap_{i=-\infty }^{\infty}\left(\left(\bigcup_{j=1}^\infty\{ H_{i+j}=H_i\}\right)\cap\left(\bigcup_{j=1}^\infty\{ H_{i-j}=H_i\}\right)\right).\end{split}\]
It is not hard to show that all these three sets are pairwise disjoint, shift-invariant, dense 
Borel subsets of $\Sigma_D$. Each ergodic element of $\mathcal M(\Sigma_D,\sigma)$ gives measure $1$ to one of these three sets (see Lemma~\ref{trichotomy}).   
One of the two ergodic maximal entropy measures gives measure $1$ to $A_\alpha$, and the other gives measure $1$ to $A_\beta$. We denote these two measures by $\mu_\alpha$ and $\mu_\beta$ so that
$\mu_\alpha(A_\alpha)=1$ and $\mu_\beta(A_\beta)=1$ (see $\S$\ref{MME-sec} for details).
For each $\gamma\in\{\alpha,\beta\}$
let
\[\mathcal M_{\gamma0}(\Sigma_D,\sigma)=\{\mu\in \mathcal M(\Sigma_D,\sigma)\colon \mu(A_\gamma\cup A_0)=1 \}.\]
The next theorem complements Theorem~\ref{break-Dyck}.

 \begin{thm}\label{rate-Dyck}
The rate function $I\colon \mathcal M(\Sigma_D)\to[0,\infty]$
of the LDP for $(\widetilde\mu_n)_{n=1}^\infty$ in Theorem~\ref{break-Dyck}
is  given by \[I(\mu)=\begin{cases}\log(M+1)-h(\mu)&\text{ if }\mu\in \mathcal M_{\alpha0}(\Sigma_D,\sigma)\cup\mathcal M_{\beta0}(\Sigma_D,\sigma),\\
    \infty&\text{ otherwise}.\end{cases}\]
\end{thm}

We prove Theorems~\ref{break-Dyck} and \ref{rate-Dyck} together, 
by unifying two LDPs for different sequences of Borel probability measures on $\mathcal M(\Sigma_D)$ constructed from subsets of $\bigcup_{n=1}^\infty{\rm Per}_n(\sigma)$.
     For each $n\in\mathbb N$ we decompose ${\rm Per}_n(\sigma)$ into three subsets:
\[\begin{split}
{\rm Per}_{\alpha,n}(\sigma)&=\{x\in {\rm Per}_n(\sigma)\colon  H_n(x)>0\};\\
{\rm Per}_{\beta,n}(\sigma)&=\{x\in {\rm Per}_n(\sigma)\colon  H_n(x)<0\};\\
{\rm Per}_{0,n}(\sigma)&=\{x\in {\rm Per}_n(\sigma) \colon H_n(x)=0\}.
\end{split}\]
Following Hamachi and Inoue \cite{HI05}, we say periodic points in $\bigcup_{n=1}^\infty{\rm Per}_{\alpha,n}(\sigma)$ (resp. $\bigcup_{n=1}^\infty{\rm Per}_{\beta,n}(\sigma)$) have {\it negative} (resp. {\it positive}) {\it multiplier},
and say periodic points in 
  $\bigcup_{n=1}^\infty{\rm Per}_{0,n}(\sigma)$ are {\it neutral}. 
The first two tie in with the two ergodic maximal entropy measures: 
periodic points with negative (resp. positive) multiplier are distributed with respect to $\mu_\alpha$ (resp. $\mu_\beta$) as their periods tend to infinity \cite{T24}.

For each $n\in\mathbb N$,
define Borel probability measures  $\widetilde\mu_{\alpha0,n}$ and $\widetilde\mu_{\beta,n}$ on $\mathcal M(\Sigma_D)$ by 
\[\widetilde\mu_{\alpha0,n}(\cdot)=
\frac{\#\{{x\in {\rm Per}_{\alpha0,n}(\sigma ) \colon V_{n}(\sigma,x)}\in\cdot\}}{\#{\rm Per}_{\alpha0,n}(\sigma ) }\]
and
\[\widetilde\mu_{\beta,n}(\cdot)=
\frac{\#\{{x\in {\rm Per}_{\beta,n}(\sigma )\colon V_{n}(\sigma,x)}\in\cdot\}}{\#{\rm Per}_{\beta,n}(\sigma ) }\]
respectively, 
where
${\rm Per}_{\alpha0,n}(\sigma )={\rm Per}_{\alpha,n}(\sigma )\cup{\rm Per}_{0,n}(\sigma )$.


 \begin{thm}\label{theorema}  
 The LDP holds for  $(\widetilde\mu_{\alpha0,n})_{n=1 }^\infty$ and $(\widetilde\mu_{\beta,n})_{n=1 }^\infty$ 
 with the rate functions $I_\alpha$ and $I_\beta$ respectively given by \[I_\alpha(\mu)=\begin{cases}\log(M+1)-h(\mu)&\text{ if }\mu\in \mathcal M_{\alpha0}(\Sigma_D,\sigma),\\
    \infty&\text{ otherwise}\end{cases}\]
    and
    \[I_\beta(\mu)=\begin{cases}\log(M+1)-h(\mu)&\text{ if }\mu\in \mathcal M_{\beta0}(\Sigma_D,\sigma),\\
    \infty&\text{ otherwise.}\end{cases}\]
 \end{thm}

In \cite{Kri74}, Krieger constructed the maximal entropy measures $\mu_\alpha$, $\mu_\beta$ by
embedding into $\Sigma_D$ two full shifts $\Sigma_\alpha$, $\Sigma_\beta$ 
over different alphabets consisting of $M+1$ symbols. 
The point is that 
Krieger's embeddings are not proper.
A proper embedding of a subshift into another means a continuous, one-to-one, not onto map
commuting with the shifts.
By the result of Hamachi and Inoue \cite{HI05}, the full shift on $M+1$ symbols cannot be properly embedded into $\Sigma_D$.

Critical structures 
in Krieger's construction are shift-invariant dense Borel subsets $K_\alpha\subsetneq\Sigma_\alpha$, 
$K_\beta\subsetneq\Sigma_\beta$ and shift-commuting homeomorphisms\footnote{As shown in Lemma~\ref{no-extension}, these homeomorphisms cannot be extended to continuous maps on the full shift spaces.} $K_\alpha\hookrightarrow\Sigma_D$, 
$K_\beta\hookrightarrow\Sigma_D$ (see $\S$\ref{DM-str}).
We exploit these structures
to prove Theorem~\ref{theorema}.
We first establish the LDP with good rate functions on the spaces $\mathcal M(K_\alpha)$ and $\mathcal M(K_\beta)$ (see Proposition~\ref{LDP-rest0})
by extending Takahashi's argument for Markov shifts \cite{Tak84}. At this point we must work out the non-compactness 
of 
$\mathcal M(K_\alpha)$ and $\mathcal M(K_\beta)$, but no tightness argument on sequences of probability measures is necessary. We use the following simple argument that relies on the {\it entropy density} of the full shift (see \cite{EKW94} and Lemma~\ref{e-dense-lem}):
for each $\gamma\in\{\alpha,\beta\}$, any shift-invariant measure in $\mathcal M(\Sigma_\gamma)\setminus\mathcal M(K_\gamma)$  arising in our constructions can be approximated in the weak* topology by ergodic elements of $\mathcal M(K_\gamma)$ with similar entropy.
We then apply the Contraction Principle to transfer the LDPs on $\mathcal M(K_\alpha)$ and $\mathcal M(K_\beta)$
to that on $\mathcal M(\Sigma_D)$ via Krieger's homeomorphisms to deduce Theorem~\ref{theorema}.
Unifying the two LDPs in Theorem~\ref{theorema} we obtain Theorems~\ref{break-Dyck} and \ref{rate-Dyck}.

\subsection{Level-1 LDPs}
Applying
the Contraction Principle (see Lemma~\ref{CP}) to the level-2 LDP in Theorem~\ref{break-Dyck}, one obtains level-1 LDPs with various continuous observables. 
Precise statements require additional definitions. 
Given a continuous function $f\colon\Sigma_D\to\mathbb R$, for each $\gamma\in\{\alpha,\beta\}$
define $I_{f,\gamma}\colon\mathbb R\to[0,\infty]$ by 
\[I_{f,\gamma}(t)=\inf\left(\left\{I(\mu)\colon\mu\in\mathcal 
 M_{\gamma0}(\Sigma_D,\sigma),\ \int f{\rm d}\mu=t \right\}\cup\{\infty\}\right),\]
 and we put
\[L_{f,\gamma}=\left\{t\in\mathbb R\colon \int f{\rm d}\mu=t\text{ for some }\mu\in\mathcal M_{\gamma0}(\Sigma_D,\sigma)\right\}.\]
Note that: $L_{f,\gamma}$ is connected and contains $\int f{\rm d}\mu_\gamma$;
 $I_{f,\gamma}(t)$ is finite if and only if $t\in L_{f,\gamma}$;
 $I_{f,\gamma}(t)=0$ if and only if $t=\int f{\rm d}\mu_\gamma$.

Besides merely continuous ones, 
 we consider observables that do not distinguish
symbols in $D_\alpha$ and do not distinguish symbols in $D_\beta$ either.
Let $\pi_D\colon x\in\Sigma_D\mapsto\pi_D(x)\in\{\alpha,\beta\}^\mathbb Z$ denote the natural projection: $\pi_D(x)$ is obtained by replacing all $\alpha_k$, $1\leq k\leq M$ in $x$ by $\alpha$ and $\beta_k$, $1\leq k\leq M$ in $x$ by $\beta$.
Let $\mathscr{F}$ denote the set of continuous functions $f\colon\Sigma_D\to\mathbb R$ for which 
there exists a H\"older continuous function $\bar f\colon\{\alpha,\beta\}^\mathbb Z\to\mathbb R$ with $f=\bar f\circ\pi_D$.
Note that elements of $\mathscr{F}$ are H\"older continuous.
\begin{thm}[the level-1 LDP]
\label{level1-thm}For any continuous function $f\colon\Sigma_D\to\mathbb R$, the LDP holds for  $(\#\left\{x\in{\rm Per}_{n}(\sigma)\colon (1/n)S_nf(x)\in \cdot\right\}/\#{\rm Per}_{n}(\sigma))_{n=1 }^\infty$.
The rate function $I_f\colon\mathbb R\to[0,\infty]$ is given by
\[I_{f}(t)=\inf\left\{I(\mu)\colon\mu\in\mathcal M(\Sigma_D),\ \int f{\rm d}\mu=t\right\}=\min\{I_{f,\alpha}(t),I_{f,\beta}(t)\}.\]
Moreover the following statements hold:
\begin{itemize}
    \item[(a)] for each $\gamma\in\{\alpha,\beta\}$, 
$I_{f,\gamma}$ is convex and continuous on $L_{f,\gamma}$;
\item[(b)]
if $f\in\mathscr{F}$, $\gamma\in\{\alpha,\beta\}$ and $L_{\gamma,f}$ is not a singleton, then 
$I_f$ is real-analytic and strictly convex 
on an open interval containing $\int f{\rm d}\mu_\gamma$. 
\end{itemize}\end{thm}

The assumption $f\in\mathscr{F}$ in Theorem~\ref{level1-thm}(b) enables us to partially relate the two functions $I_{f,\alpha}$, $I_{f,\beta}$ to certain multifractal spectra of $(\Sigma_\alpha,\sigma_\alpha)$ and 
$(\Sigma_\beta,\sigma_\beta)$ respectively
(see Lemma~\ref{U-lemma}). We use the thermodynamic formalism \cite{Rue04} to show the analyticity of these two spectra.
Part of the spectra are transferred back to the Dyck shift via Krieger's homeomorphisms, and related to $I_f$. The asserted strict convexity follows from a general power series argument (see Lemma~\ref{convex-lem}) building on the analyticity.

In the dynamical systems setup, the lack of smoothness in large deviations rate functions should reflect singularities of the underlying dynamical system. 
Since $I_f$ is a superposition of the convex continuous functions $I_{f,\alpha}$ and $I_{f,\beta}$ as in  Theorem~\ref{level1-thm}, it can have points of non-differentiability
when $f\in\mathscr{F}$.
As an illustration, consider $f\in\mathscr{F}$ given by \begin{equation}\label{phi-0}f(x)=\begin{cases}0&\text{ if }x_0\in D_\alpha,\\ 1&\text{ if }x_0\in D_\beta.\end{cases}\end{equation}
The corresponding level-1 rate function $I_{f}\colon\mathbb R\to[0,\infty]$ is given by
\begin{equation}\label{ratef-1}I_f(t)=\begin{cases}t\log t+(1-t)\log(1-t)+\log((M+1)/M^{1-t})  & \text{if }t\in[0,1/2],\\
t\log t+(1-t)\log(1-t)+\log((M+1)/M^t) &\text{if }t\in(1/2, 1],\\
\infty & \text{if }t\in\mathbb R\setminus[0,1].\end{cases}\end{equation}
See \textsc{Figure}~\ref{fig1} for the graph of $I_f$ on its effective domain. At $1/2$, $I_f$ is not differentiable.
In contrast, 
 level-1 rate functions for H\"older continuous observables on topologically mixing Markov shifts are globally analytic\footnote{
Although we could not locate any literature,  this assertion can be proved, for instance using the thermodynamic formalism as we do in $\S$\ref{last-sec}.}.
A proof of \eqref{ratef-1} is given in $\S$\ref{ratepf}.

\begin{figure}
\begin{center}
\includegraphics[height=16.5cm,width=10cm]{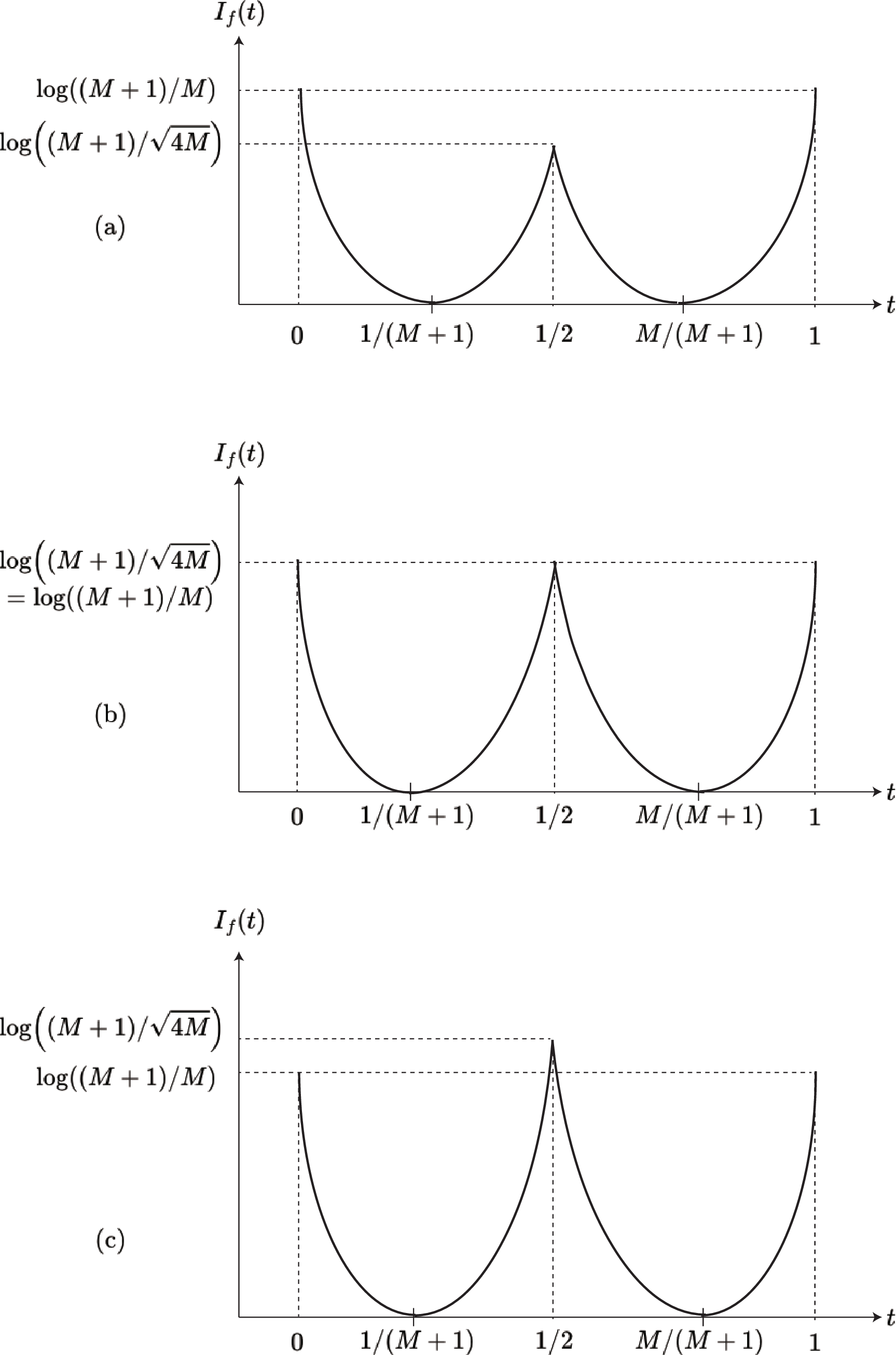}
\caption
{ The graph of the rate function $I_f$ on its effective domain $[0,1]$ for $f$ in \eqref{phi-0}; (a) $M\leq 3$; (b) $M=4$; (c) $M\geq5$.}\label{fig1}
\end{center}
\end{figure}

 \subsection{Generalizations and organization of the paper}
To keep the reasonable length of this paper we have omitted generalizations to other subshifts. 
The main results can be generalized with no essential difficulty to include the Motzkin shift, another representative of property A subshifts whose admissible words are legally aligned brackets with interspersed unit elements. For a precise definition of the Motzkin shift, see \cite{M04}.

The rest of this paper consists of two sections and one appendix. In $\S$2 we  summarize  preliminary materials including known results on the Dyck shift in \cite{HI05,Kri74}. In $\S$3 we prove the main results stated above. In Appendix~A we prove a lemma on the Dyck shift that motivates the above simple approximation argument.

 \section{Preliminaries}
This section summarizes preliminary materials needed for the proofs of the main results.
In $\S$\ref{weak*} we review the notion of weak* topology. In $\S$\ref{notation} we introduce basic notations on subshifts that permeate the paper, and prove a lemma on variations of Birkhoff averages of continuous functions.
The remaining subsections 
summarize known results on the Dyck shift, with main references to \cite{HI05} and \cite{Kri74}.

\subsection{Weak* topology}\label{weak*}
Let $X$ be a topological space and let $C(X)$ denote the space of real-valued continuous functions on $X$ endowed with the supremum norm $\|\cdot\|$. For $f\in C(X)$
and $\mu\in \mathcal M(X)$ write $\langle f,\mu\rangle=\int f{\rm d}\mu$.
The {\it weak* topology} on $\mathcal M(X)$ is the coarsest topology that makes  $\mu\in\mathcal M(X)\to\langle f,\mu\rangle$ continuous for any  $f\in C(X)$. In this topology, a sequence $(\mu_n)_{n=1}^\infty$ of elements of $\mathcal M(X)$ converges to $\mu\in\mathcal M(X)$ if and only if $\lim_{n\to\infty}\langle f,\mu_n\rangle=\langle f,\mu\rangle$ holds for any $f\in C(X)$. This condition is equivalent to $\lim_{n\to\infty}\langle f,\mu_n\rangle=\langle f,\mu\rangle$ for any $f\in C(X)$ that is uniformly continuous  (see \cite[Chapter~9]{Str25}).
The relative topology on $\mathcal M(Y)$
with respect to the weak* topology on $\mathcal M(X)$ coincides with the coarsest topology that makes
$\mu\in\mathcal M(Y)\to\langle f,\mu\rangle$ continuous for any $f\in C(X)$.

\begin{lemma}\label{top-lem}Let $X$ be a metric space and $Y$ a Borel subset of $X$. The weak* topology on $\mathcal M(Y)$ coincides with the relative topology on $\mathcal M(Y)$ with respect to the weak* topology on $\mathcal M(X)$.\end{lemma}
\begin{proof}Any bounded uniformly continuous function $Y$ can be extended (uniquely) to a continuous function on the closure of $Y$. By Tietze's theorem, this function can be extended to a bounded continuous function on $X$. \end{proof}

\subsection{Words, cylinders, bounded  variations of Birkhoff averages}\label{notation}
Let $S$ be a finite alphabet.
A finite string of elements of $S$ is called a {\it word} in $S$.
   Let $\Sigma$ be a subshift over $S$.
   Let $L(\Sigma)$ denote the set of words in $S$ that appear in some elements of $\Sigma$.
   Elements of $L(\Sigma)$ are called {\it admissible words} for $\Sigma$. For each $n\in\mathbb N$ let $L_n(\Sigma)$ denote the set of admissible words for $\Sigma$ with word length $n$.
   For $j\in\mathbb Z$, $n\in\mathbb N$, $\omega=\omega_1\cdots\omega_n\in L_n(\Sigma)$, define a {\it cylinder set}, or {\it $n$-cylinders} by
 \[\Sigma(j,\omega)=\{(x_i)_{i\in\mathbb Z}\in\Sigma\colon x_i=\omega_{i-j+1}\text{ for }j\leq i\leq  j+n-1\}.\]
 For simplicity, we write $\Sigma(0,\omega)=[\omega]$ when the context is clear.
 For each $i\in\mathbb Z$ and $A\subset L(\Sigma)$, we set \[\Sigma(i,A)=\bigcup_{a\in A}\Sigma(i,a).\]


The next lemma is an easy consequence of the topology on subshifts, but we include a proof for completeness.
\begin{lemma}\label{variation}
Let $\Sigma$ be a subshift and let $\tau\colon\Sigma\to\Sigma$ be the left shift.
For any $f\in C(\Sigma)$ and any $\varepsilon>0$ there exists $N\in\mathbb N$ such that
for every integer $n\geq N$ and every $\omega\in L_n(\Sigma)$, we have
\[\max_{x,y\in[\omega]}\frac{1}{n}\left|\sum_{i=0}^{n-1}f(\tau^{i}x)-\sum_{i=0}^{n-1}f(\tau^{i}y)\right|< \varepsilon.\]
\end{lemma}
\begin{proof}
The topology on $\Sigma$ is metrizable by the metric given by $d(x,y)=2^{-s(x,y)}$, $s(x,y)=\min\{|i|\colon i\in\mathbb Z,\ x_i\neq y_i\}$ for distinct points $x,y\in\Sigma$.
Since $f$ is uniformly continuous, for any $\varepsilon>0$ there exists $\delta>0$ such that if $x,y\in\Sigma$ and $d(x,y)<\delta$
then $|f(x)-f(y)|<\varepsilon/2$.
Fix $k\in\mathbb N$ with $2^{-k}<\delta$. Let $n\geq 2k$ be an integer and let $\omega\in L_n(\Sigma)$. 
 For all $x,y\in[\omega]$ we have 
\[\begin{split}\frac{1}{n}\left|\sum_{i=0}^{n-1}f(\tau^{i}x)-\sum_{i=0}^{n-1}f(\tau^{i}y)\right|\leq&\frac{1}{n}\sum_{i=0}^{k-1}|f(\tau^{i}x)-f(\tau^{i}y)|\\
&+\frac{1}{n}\sum_{i=k}^{n-k-1}|f(\tau^{i}x)-f(\tau^{i}y)|\\
&+\frac{1}{n}\sum_{i=n-k}^{n-1}|f(\tau^{i}x)-f(\tau^{i}y)|\\
\leq& \frac{4k\|f\|}{n}+\frac{n-2k}{2n}\varepsilon< \varepsilon,\end{split}\]where the last inequality holds if $n$ is sufficiently large. \end{proof}

\subsection{Borel embeddings of full shifts into the Dyck shift}\label{DM-str}


Following \cite{Kri74}
we delve into the structure of the Dyck shift $\Sigma_D$.
Let
\[\begin{split}
B_\alpha&=\bigcap_{i=-\infty}^\infty\left(
\Sigma_D(i,D_\alpha)\cup\left(
\Sigma_D(i,D_\beta)\cap\bigcup_{j=1}^\infty \{H_{i-j+1}=H_{i+1}\}\right)\right)\end{split}\]
and
\[\begin{split}
B_\beta&=\bigcap_{i=-\infty}^\infty\left(\left(\Sigma_D(i,D_\alpha)\cap \bigcup_{j=1}^\infty \{H_i=H_{i+j}\}\right)\cup\Sigma_{D}(i,D_\beta)\right).\end{split}\]
 The set $B_\alpha$ (resp. $B_\beta$)
is precisely the set of sequences in $\Sigma_D$ such that any right (resp. left) bracket in the sequence is closed. 
One can check that 
they are shift-invariant and satisfy
\begin{equation}\label{subset-AB}A_\alpha\cup A_0\subset B_\alpha\ \text{ and }\ A_\beta\cup A_0\subset B_\beta,\  A_\alpha\cap B_\beta=\emptyset \ \text{ and } \ A_\beta\cap B_\alpha=\emptyset.\end{equation}

 We introduce two full shifts over different alphabets consisting of $M+1$ symbols:
 \[\Sigma_\alpha=(D_\alpha\cup\{\beta\})^{\mathbb Z}\quad\text{and}\quad\Sigma_\beta=(\{\alpha\}\cup D_\beta)^{\mathbb Z}.\]
 Let $\sigma_\alpha$, $\sigma_\beta$ denote the left shifts acting on $\Sigma_\alpha$, $\Sigma_\beta$ respectively.
Define 
 $\phi_\alpha\colon B_\alpha\to \Sigma_\alpha$ 
by
\[(\phi_\alpha(x))_i=\begin{cases}
  \beta&\text{ if }x_i\in D_\beta,\\
 x_i &\text{ otherwise.}
\end{cases}\]
In other words, $\phi_\alpha(x)$ is obtained by replacing all $\beta_k$, $1\leq k\leq M$ in $x$ by $\beta$.
Clearly $\phi_\alpha$ is continuous.
Similarly, define $\phi_\beta\colon B_\beta\to \Sigma_\beta$ by
\[(\phi_\beta(x))_i=\begin{cases}
    \alpha&\text{ if }x_i\in D_\alpha,\\
    x_i&\text{ otherwise.}
\end{cases}\]
In other words, $\phi_\beta(x)$ is obtained by replacing all $\alpha_k$, $1\leq k\leq M$ in $x$ by $\alpha$.
Clearly $\phi_\beta$ is continuous too. We set
\[K_\alpha=\phi_\alpha(B_\alpha)\ \text{ and }\ K_\beta=\phi_\beta(B_\beta).\]

    For each $i\in\mathbb Z$ define $ H_{\alpha,i}\colon \Sigma_\alpha\to\mathbb Z$ by       
      \[\begin{split}H_{\alpha,i}(y)&=\begin{cases}\sum_{j=0}^{i-1} \sum_{k=1}^{M}(\delta_{{\alpha_k},y_j}-\delta_{\beta,y_j})&\text{ for }  i\geq1,\\\sum_{j=i}^{-1} \sum_{k=1}^{M}(\delta_{{\beta},y_j}-\delta_{\alpha_k,y_j})&\text{ for } i\leq -1,\\
      0&\text{ for }i=0.\end{cases}\end{split}\]
We now define $\psi_\alpha\colon K_\alpha\to D^\mathbb Z$ by
\[(\psi_\alpha(y))_i=\begin{cases}
  \beta_k&\text{ if }y_i=\beta,\ y_{s_\alpha(i,y)}=
  \alpha_k,\ 1\leq k\leq M,\\
   y_i&\text{ otherwise,}
\end{cases}\]
where \[s_\alpha(i,y)=\max\{j<i+1\colon  H_{\alpha,j}(y)= H_{\alpha,i+1}(y)\}.\]
Clearly $\psi_\alpha$ is continuous. Similarly, for each $i\in\mathbb Z$ we define 
$H_{\beta,i}\colon \Sigma_\beta\to\mathbb Z$ by 
      \[\begin{split}
            H_{\beta,i}(y)&=\begin{cases}\sum_{j=0}^{i-1} \sum_{k=1}^{M}(\delta_{{\alpha},y_j}-\delta_{\beta_k,y_j})&\text{ for }  i\geq1,\\\sum_{j=i}^{-1} \sum_{k=1}^{M}(\delta_{{\beta_k},y_j}-\delta_{\alpha,y_j})&\text{ for } i\leq -1,\\
      0&\text{ for }i=0.\end{cases}\end{split}\]
We also define $\psi_\beta\colon K_\beta\to D^\mathbb Z$ by
\[(\psi_\beta(y))_i=\begin{cases}
   \alpha_k&\text{ if }y_i=\alpha,\ y_{s_\beta(i,y)}=\beta_k,\ 1\leq k\leq M,\\
    y_i&\text{ otherwise, }
\end{cases}\]
where \[s_\beta(i,y)=\min\{j>i\colon  H_{\beta,j}(y)= H_{\beta,i}(y)\}-1.\]
Clearly $\psi_\beta$ is continuous too.

Let us summarize the objects introduced above in the following diagram:
 \[\begin{split}
  \xymatrix{
    & B_\alpha\cup   B_\beta \ar@<-0.8ex>[ld]_{\phi_\alpha } 
 \ar@<0.8ex>[rd]^{\phi_\beta } &  \\
 \Sigma_\alpha\supsetneq
 K_\alpha \ar[ru]_{\psi_\alpha}\ \ \   &     &  \ \ \   
 \ar[lu]^{\psi_\beta} 
 K_\beta\subsetneq
 \Sigma_\beta.}\end{split}\]
 For each $\gamma\in\{\alpha,\beta\}$,
$\psi_\gamma(K_\gamma)=B_\gamma$, $\psi_\gamma$ is a homeomorphism whose inverse is $\phi_\gamma$, and
$\phi_\gamma\circ\sigma|B_\gamma=\sigma_\gamma\circ\phi_\gamma$. 
 Since there is no proper embedding of the full shift on $M+1$ symbols into the Dyck shift 
\cite[Theorem~5.3]{HI05}, $K_\gamma$ is a proper subset of $\Sigma_\gamma$.
Since the Bernoulli measure on $\Sigma_\alpha$ (resp. $\Sigma_\beta$) associated with the probability vector $(1/(M+1),\ldots,1/(M+1))$ gives measure $1$ to $[\beta]$ (resp. $[\alpha]$) by Lemma~\ref{gyak-lem} below,
$K_\alpha$ (resp. $K_\beta$) is dense in $\Sigma_\alpha$ (resp. $\Sigma_\beta$).

\begin{lemma}[\cite{Kri74}, pp.102--103]
\label{trichotomy}
If $\mu\in \mathcal M(\Sigma_D,\sigma)$ is ergodic, then either 
$\mu(A_\alpha)=1$,  $\mu(A_\beta)=1$ or $\mu(A_0)=1$. 
\end{lemma}

Elements of $\mathcal M(\Sigma_\gamma)$ that give measure $1$
to $K_\gamma$ can be transported via $\psi_\gamma$ 
to elements of $\mathcal M(\Sigma_D)$.
 The next lemma, proved in \cite[Section~4]{Kri74},  gives a sufficient condition for ergodic elements of $\mathcal M(\Sigma_\gamma,\sigma_\gamma)$ to be transported
to elements of $\mathcal M(\Sigma_D)$.
For completeless
we include a proof.

\begin{lemma}
\label{gyak-lem}
\
\begin{itemize}
    \item[(a)] If $\nu\in \mathcal M(\Sigma_\alpha,\sigma_\alpha)$ is ergodic and $\nu([\beta])<1/2$ then $\nu(K_\alpha)=1$.
    \item[(b)] If $\nu\in \mathcal M(\Sigma_\beta,\sigma_\beta)$ is ergodic and $\nu([\alpha])<1/2$ then $\nu(K_\beta)=1$.
\end{itemize}
\end{lemma}

\begin{proof}
By the definition of $B_\alpha$ and that of $\phi_\alpha$,
we have
\[\Sigma_\alpha\setminus K_\alpha=\bigcup_{i=-\infty}^\infty\bigcap_{j=1}^\infty(\Sigma_\alpha\setminus\{ H_{\alpha,i-j+1}= H_{\alpha,i+1}\})\]
where $\{ H_{\alpha,i-j+1}=H_{\alpha,i+1}\}=\{x\in\Sigma_\alpha\colon H_{\alpha,i-j+1}(x)=H_{\alpha,i+1}(x)\}$.
For each $x\in \Sigma_\alpha\setminus
K_\alpha$, 
there exists $i\in\mathbb Z$ such that
$x_i=\beta$ and
$ H_{\alpha,i-j+1}(x)\neq H_{\alpha,i+1}(x)$ for every $j\geq1$.
By induction, for every $j\geq1$ we have
\[\#\{i-j+1\leq k\leq i\colon x_k=\beta\}>\#\{i-j+1\leq k\leq i\colon x_k\neq\beta\},\]
and so
\[\liminf_{j\to\infty}\frac{1}{j}\cdot\#\{i-j+1\leq k\leq i\colon x_k=\beta\}\geq\frac{1}{2}.\]
If $\nu\in \mathcal M(\Sigma_\alpha,\sigma_\alpha)$ is ergodic and satisfies $\nu([\beta])<1/2$, then
Birkhoff's ergodic theorem applied to $(\sigma_\alpha^{-1},\nu)$ yields $\nu(K_\alpha)=1$, as required in (a). A proof of (b) is analogous. 
\end{proof}

\begin{remark}\label{erg-rem}
The assumption of ergodicity cannot be dropped from Lemma~\ref{gyak-lem}.
For example, take $\nu_0\in\mathcal M(\Sigma_\alpha,\sigma_\alpha)$ 
that gives measure $1$ to $D_\alpha^\mathbb Z$.
Let $\beta^\infty$ denote the fixed point of $\sigma_\alpha$ in the cylinder set $[\beta]$. Let $0<t<1/2$ and
consider the measure $\nu=(1-t)\cdot\nu_0+t\cdot\nu_{\beta^\infty}$.
Then $\nu([\beta])=t<1/2$ and $\nu(K_\alpha)=(1-t)<1$. \end{remark}

\begin{remark}\label{view-rem}
Let $\gamma\in\{\alpha,\beta\}$.
Since $\sigma_\gamma(K_\gamma)=K_\gamma$, 
we may identify elements of $\mathcal M(K_\gamma,\sigma_\gamma|K_\gamma)$
as elements of $\mathcal M(\Sigma_\gamma,\sigma_\gamma)$.
This identification preserves entropy and ergodicity of invariant measures.
\end{remark}

\begin{lemma}\label{lem-half}\
\begin{itemize}\item[(a)] For any $\nu\in \mathcal M(K_\alpha,\sigma_\alpha|K_\alpha)$, we have $\nu([\beta])\leq1/2.$
\item[(b)] For any $\nu\in \mathcal M(K_\beta,\sigma_\beta|K_\beta)$, we have $\nu([\alpha])\leq1/2.$
\end{itemize}\end{lemma}
\begin{proof} Suppose $\nu\in \mathcal M(K_\alpha,\sigma_\alpha|K_\alpha)$ is ergodic. Since $A_\alpha\cup A_0\subset B_\alpha$ and $A_\beta\cap B_\alpha=\emptyset$ as in \eqref{subset-AB}, the ergodic measure $\nu\circ\psi_\alpha^{-1}$ gives measure $1$ to either $A_\alpha$ or $A_0$ by Lemma~\ref{trichotomy}.
We claim $\nu\circ\psi_\alpha^{-1}([D_\beta])\leq 1/2$, for otherwise
 by Birkhoff's ergodic theorem we would have $\lim_{i\to\infty}H_i(x)\to-\infty$ and $\lim_{i\to-\infty}H_i(x)\to\infty$ for $\nu\circ\psi_\alpha^{-1}$-a.e., namely $\nu\circ\psi_\alpha^{-1}(A_\beta)=1$.
 Hence we obtain $\nu([\beta])=\nu\circ\psi_\alpha^{-1}([D_\beta])\leq1/2.$
 If $\nu\in \mathcal M(K_\alpha,\sigma_\alpha|K_\alpha)$ is not ergodic, then the ergodic decomposition theorem yields $\nu([\beta])\leq1/2$ as required in (a). 
  A proof of (b) is analogous. \end{proof}

\subsection{Maximal entropy measures for the Dyck shift}\label{MME-sec} 
For each $\gamma\in\{\alpha,\beta\}$,
let $\nu_\gamma$ denote the Bernoulli measure on $\Sigma_\gamma$ associated with the probability vector $(1/(M+1),\ldots,1/(M+1))$. Since $\nu_\alpha([\beta])=1/(M+1)<1/2$, Lemma~\ref{gyak-lem}(a) gives $\nu_\alpha(K_\alpha)=1$. Since $\nu_\beta([\alpha])=1/(M+1)<1/2$, Lemma~\ref{gyak-lem}(b) gives $\nu_\beta(K_\beta)=1$.
Hence, the measures  \begin{equation}\label{ergmme}\mu_{\alpha}=\nu_\alpha\circ\psi_\alpha^{-1}\ \text{ and } \ \mu_\beta=\nu_\beta\circ\psi_\beta^{-1}\end{equation}
are in $\mathcal M(\Sigma_D,\sigma)$
and of entropy $\log(M+1)$.
 From direct calculations based on \eqref{ergmme}, 
we deduce the following identities for $1\leq k\leq M$:
 \begin{equation}\begin{split}\label{balance2}
 \mu_\alpha
\left([\alpha_k]\right)&=\mu_\alpha
\left([D_\beta]\right)=\frac{1}{M+1};\\
\mu_\beta
\left([D_\alpha]\right)&=
\mu_\beta\left([\beta_k]\right)=\frac{1}{M+1 }.\end{split}\end{equation}
In particular, $\mu_\alpha\neq\mu_\beta$ holds.
By virtue of the trichotomy in Lemma~\ref{trichotomy}, the ergodicity of $\mu_\alpha$, $\mu_\beta$ and \eqref{balance2}  altogether imply
 \begin{equation}\label{measure-nu}\mu_\alpha(A_\alpha )=1\ \text{ and }\ 
 \mu_\beta(A_\beta )=1.\end{equation}
 From \eqref{subset-AB}, Lemma~\ref{trichotomy} and \eqref{measure-nu}
 it follows that $\mu_\alpha$, $\mu_\beta$ are precisely the ergodic measures of maximal entropy $\log(M+1)$.


\subsection{Estimates of numbers of periodic points}\label{count-sec}
Hamachi and Inoue \cite{HI05} obtained  exact formulas for numbers of periodic points of the Dyck shift.
 The next two lemmas essentially follow from their results in \cite{HI05}.
\begin{lemma}[\cite{T24}, Lemma~2.6]\label{per-lem} For each $\gamma\in\{\alpha,\beta\}$ we have \[\frac{1}{3}(M+1)^n\leq\#{\rm Per}_{\gamma,n}(\sigma )<(M+1)^n\]
for all sufficiently large integer $n\geq1$.\end{lemma}

\begin{lemma}\label{per-eq0}
For every $n\in\mathbb N$ we have
\[\#{\rm Per}_{0,n}(\sigma )=\begin{cases}\vspace{1mm}0&\ \text{ if $n$ is odd,}\\\displaystyle{\binom{n}{n/2 }M^{n/2}}\leq (2\sqrt{M})^n&\ \text{ if $n$ is even.}\end{cases}\]\end{lemma}
\begin{proof} For each $n\in\mathbb N$, there is a one-to-one correspondence between ${\rm Per}_{0,n}(\sigma)$ and  $\{\omega\in L_n(\Sigma_D)\colon{\rm red}(\omega)=1\}$. If $n$ is odd, then these sets are empty. If $n$ is even, then there are $\binom{n}{n/2 }$ number of ways to distribute positions of symbols in $D_\alpha$
in $n$ strings. For each such a way, there are $(\#D_\alpha)^{n/2}$ number of ways to distribute symbols in $D_\alpha$. For each such a way, the distribution of symbols in $D_\beta$ in free positions is determined uniquely. 
This yields the desired equality.\end{proof}

 \section{The LDP for periodic points of the Dyck shift}
In this section we complete the proofs of the main results. They 
rely on the level LDP on the spaces $\mathcal M(K_\alpha)$ and $\mathcal M(K_\beta)$ stated as
 Proposition~\ref{LDP-rest0} in $\S$\ref{restricted}. In $\S$\ref{lower-sec} and $\S$\ref{upper-sec} we derive intermediate large deviations bounds, and use them to prove Proposition~\ref{LDP-rest0} in $\S$\ref{pfprop}.
  We prove Theorem~\ref{theorema} in $\S$\ref{pfthmc}, Theorems~\ref{break-Dyck} and \ref{rate-Dyck} in $\S$\ref{pf-final}, Theorem~\ref{level1-thm} in $\S$\ref{last-sec}, and \eqref{ratef-1} in $\S$\ref{ratepf}.


 \subsection{The level-2 LDP on two embedded full shifts}\label{restricted}
 For each $n\in\mathbb N$, define a Borel probability measure
  $\widetilde\nu_{\alpha0,n}$ on $\mathcal M(K_\alpha)$ by 
\[\widetilde\nu_{\alpha0,n}(\cdot)=
\frac{\#\{{x\in \phi_\alpha({\rm Per}_{\alpha0,n}(\sigma )) \colon V_{n}(\sigma_\alpha,x)}\in\cdot\}}{\#\phi_\alpha({\rm Per}_{\alpha0,n}(\sigma ) )}.\]
We also define a Borel probability measure
  $\widetilde\nu_{\beta,n}$ on $\mathcal M(K_\beta)$ by 
\[\widetilde\nu_{\beta,n}(\cdot)=
\frac{\#\{{x\in \phi_\beta({\rm Per}_{\beta,n}(\sigma ))\colon V_{n}(\sigma_\beta,x)}\in\cdot\}}{\#\phi_\beta({\rm Per}_{\beta,n}(\sigma ) )}\]

\begin{prop}\label{LDP-rest0} The LDP holds for 
$(\widetilde\nu_{\alpha0,n}  )_{n=1}^\infty$ and $(\widetilde\nu_{\beta,n} )_{n=1}^\infty$ with 
the good rate functions
  $J_\alpha\colon \mathcal M(K_\alpha)\to [0,\infty]$ and $J_\beta\colon \mathcal M(K_\beta)\to [0,\infty]$ respectively given by
\[J_\alpha(\nu)=\begin{cases}\log(M+1)-h(\nu)\ &\text{ if $\nu\in \mathcal M(K_\alpha,\sigma_\alpha|K_\alpha)$},\\
\infty&\text{ otherwise,}\end{cases}\]
and
\[J_\beta(\nu)=\begin{cases}\log(M+1)-h(\nu)\ &\text{ if $\nu\in \mathcal M(K_\beta,\sigma_\beta|K_\beta)$},\\
\infty&\text{ otherwise.}\end{cases}\]
\end{prop}

We will deduce the lower and upper  bounds of the LDPs in Proposition~\ref{LDP-rest0} from the corresponding intermediate lower and upper bounds for particular open and closed sets.

A common key ingredient in the deduction of both bounds is the following notion.
We say a subshift $\Sigma$ is {\it entropy dense} if
for any shift-invariant measure $\mu\in\mathcal M(\Sigma)$ that is not ergodic, any $\varepsilon>0$ and any open subset $\mathcal G$ of $\mathcal M(\Sigma)$ that contains $\mu$, there exists a shift-invariant measure $\mu'\in \mathcal G$ such that 
$h(\mu')>h(\mu)-\varepsilon$.
Eizenberg et al. proved the entropy density in a great generality
\cite[Proposition~6.1]{EKW94}. 
The next lemma is a particular case of their result.
\begin{lemma}\label{e-dense-lem} 
The full shift is entropy dense.
\end{lemma}

To proceed we need some notation. 
Let $\gamma\in\{\alpha,\beta\}$.
For $\ell\in\mathbb N$ let
 \[C(\Sigma_\gamma)^\ell=\{\vec F=(f_1,\ldots,f_\ell)\colon f_j\in C(\Sigma_\gamma)\ \text{ for } 1\leq j\leq\ell\}.\]
 For $\vec F=(f_1,\ldots,f_\ell)\in C(\Sigma_\gamma)^\ell$, $\vec \Lambda=(\lambda_1,\ldots,\lambda_\ell)\in\mathbb R^\ell$
and $\nu\in \mathcal M(\Sigma_\gamma)$, the expression 
$\langle\vec F,\nu\rangle>\vec{\Lambda}$ indicates that $\langle f_j,\nu\rangle>\lambda_j$ holds for every $ 1\leq j\leq\ell$. Write
$\|\vec F\|=\max_{1\leq j\leq \ell}\|f_j\|$. 
For $\varepsilon>0$ write $\vec \varepsilon=(\varepsilon,\ldots,\varepsilon)\in\mathbb R^\ell.$
For a subset $\Omega$ of $L(\Sigma_\gamma)$, let
\[[\Omega]=\bigcup_{\omega\in\Omega}[\omega].\]

\subsection{Intermediate lower bounds}\label{lower-sec} The next two lemmas give large deviations lower bounds for particular open sets.
\begin{lemma}\label{lowper0}
Let $\ell\in\mathbb N$, $\vec{F}\in C(\Sigma_\alpha)^\ell$, $\vec \Lambda\in\mathbb R^\ell$,
and let $\mathcal G$ be a non-empty open subset of $\mathcal M(K_\alpha)$ of the form
\[\mathcal G=\{\nu\in  \mathcal M(K_\alpha)\colon \langle\vec F,\nu\rangle>\vec \Lambda\}.\]
For any $\nu\in \mathcal G$ we have 
\[\liminf_{n\to\infty}\frac{1}{n}\log\widetilde\nu_{\alpha0,n}( \mathcal G)\geq -J_\alpha(\nu).\]
 \end{lemma}

 \begin{lemma}\label{lowper0'}
Let $\ell\in\mathbb N$, $\vec{F}\in C(\Sigma_\beta)^\ell$, $\vec \Lambda\in\mathbb R^\ell$,
and let $\mathcal G$ be a non-empty open subset of $\mathcal M(K_\beta)$ of the form
\[\mathcal G=\{\nu\in  \mathcal M(K_\beta)\colon \langle\vec F, \nu\rangle>\vec \Lambda\}.\]
For any $\nu\in \mathcal G$ we have 
 \[
\liminf_{n\to\infty}\frac{1}{n}\log\widetilde\nu_{\beta,n}(\mathcal G)\geq -J_\beta(\nu).\]
 \end{lemma}

Below we only give a proof of Lemma~\ref{lowper0}  since that of Lemma~\ref{lowper0'} is analogous.

\begin{proof}[Proof of Lemma~\ref{lowper0}]
First we show that it suffices to prove the desired inequality only for ergodic measures. Then we prove this inequality for each  ergodic measure. 

 Let $\nu\in \mathcal G$.  
 If $\nu\notin \mathcal M(K_\alpha,\sigma_\alpha|K_\alpha)$
  then $J_\alpha(\nu)=\infty$  and the desired inequality holds trivially.
 Suppose $\nu\in \mathcal M(K_\alpha,\sigma_\alpha|K_\alpha)$. As in Remark~\ref{view-rem}, we regard $\nu$ as an element of $\mathcal M(\Sigma_\alpha,\sigma_\alpha)$.
 Suppose that $\nu$ is not ergodic with respect to $\sigma_\alpha$. 
Since the indicator function of the cylinder set  
$[\beta]$ and all components of $\vec F$ are continuous functions, 
 by Lemma~\ref{e-dense-lem},
 for any $\varepsilon>0$
 there exists an ergodic measure $\nu'\in  \mathcal M(\Sigma_\alpha,\sigma_\alpha)$ such that \begin{equation}\label{Er-I}h(\nu')>h(\nu)-\varepsilon,\end{equation}
 \begin{equation}\label{Er-II}
  \nu'([\beta]) <\nu([\beta])+\varepsilon \ \text{ and }\ 
  \langle\vec F,\nu'\rangle>\langle \vec F,\nu\rangle-\vec\varepsilon.\end{equation}
The first inequality in \eqref{Er-II} and Lemma~\ref{lem-half} give
  \begin{equation}\label{Er-IV}\nu'([\beta]) <\frac{1}{2}+\varepsilon.\end{equation}
  In particular, $\nu'$ may give measure $1/2$ to $[\beta]$ and so may not give
   measure $1$ to $K_\alpha$. We construct another shift-invariant ergodic measure that gives measure $1$ to $K_\alpha$ and is contained in $\mathcal G$. 
Let 
$\alpha_1^\infty$ denote the fixed point of $\sigma_\alpha$ in the cylinder set $[\alpha_1]$. 
Consider the convex combination
\[\nu''=\frac{1}{1+3\varepsilon}\cdot\nu'+\frac{3\varepsilon}{1+3\varepsilon}\cdot\delta_{\alpha_1^\infty}\in \mathcal M(\Sigma_\alpha,\sigma_\alpha).\]
We have
\begin{equation}\label{erg-II''}h(\nu'')=\frac{1}{1+3\varepsilon}\cdot h(\nu),\end{equation}
and by \eqref{Er-IV},
\begin{equation}\label{erg-II'}\nu''([\beta])=\frac{1}{1+3\varepsilon}\cdot\nu'([\beta])\leq\frac{1}{1+3\varepsilon}\cdot\left(\frac{1}{2}+\varepsilon\right)<\frac{1}{2},\end{equation}
and by
 the second inequality in \eqref{Er-II}, 
\begin{equation}\label{erg-II} \begin{split}
 \langle\vec F,\nu''\rangle&\geq\frac{1}{1+3\varepsilon}\cdot\langle\vec F,\nu'\rangle-\frac{3\|\vec F\|\varepsilon}{1+3\varepsilon}\cdot\vec 1\\&\geq\frac{1}{1+3\varepsilon}\cdot\langle\vec F,\nu\rangle-\frac{3\| \vec F\|\varepsilon+\varepsilon}{1+3\varepsilon}\cdot\vec 1\\
 &\geq\langle \vec F,\nu\rangle-\frac{6\|\vec F\|\varepsilon+\varepsilon}{1+3\varepsilon}\cdot\vec 1.\end{split}\end{equation}
  If $\varepsilon>0$ is small enough so that 
 $\langle\vec F,\nu\rangle-\frac{6\|\vec F\|\varepsilon+\varepsilon}{1+3\varepsilon}\cdot\vec 1>\vec \Lambda$,
then by \eqref{erg-II''}, \eqref{erg-II'}, \eqref{erg-II} and Lemma~\ref{e-dense-lem} applied to $\nu''$, there exists an ergodic measure $\nu'''\in \mathcal M(\Sigma_\alpha,\sigma_\alpha)$ such that 
\begin{equation}\label{erg-IIIII'}h(\nu''')\geq\frac{1}{1+4\varepsilon}\cdot h(\nu),\end{equation}
 \begin{equation}\label{erg-III} \nu'''([\beta])<\frac{1}{2} \ \text{ and } \
 \langle\vec F,\nu'''\rangle>\vec \Lambda.\end{equation}
 By Lemma~\ref{gyak-lem} and
 the first inequality in \eqref{erg-III}
 we have $\nu'''(K_\alpha)=1$,  and so from the second inequality in \eqref{erg-III} we obtain
$\nu'''\in \mathcal G$.
Since $\varepsilon>0$ is arbitrary, by \eqref{erg-IIIII'}
  it suffices to show the desired inequality in Lemma~\ref{lowper0} for any element of $\mathcal G\cap\mathcal M(K_\alpha,\sigma_\alpha|K_\alpha)$ that is ergodic.

Let
$\nu\in \mathcal G\cap\mathcal M(K_\alpha,\sigma_\alpha|K_\alpha)$ be ergodic, and 
let $\varepsilon>0$ satisfy
\begin{equation}\label{lower-e10}\langle\vec F,\nu\rangle-6\|\vec F\|\cdot\vec\varepsilon-\vec\varepsilon>\vec \Lambda.\end{equation}
As in Remark~\ref{view-rem}, we regard $\nu$ as an element of $\mathcal M(\Sigma_\alpha,\sigma_\alpha)$.
Since the collection of $1$-cylinders in $\Sigma_\alpha$ forms a generator, 
by Birkhoff's ergodic theorem and Shannon-McMillan-Breiman's theorem, for all sufficiently large $n>1/(3\varepsilon)$ there exists a subset $E_n$ of $L_n(\Sigma_\alpha)$
such that 
\begin{equation}\label{lower-e20}\#E_n\geq\exp((h(\nu)-\varepsilon)n),\end{equation}
and for all $x\in\bigcup_{\omega\in E_n}[\omega]$,
\begin{equation}\label{lower-e-13} V_n(\sigma_\alpha,x)([\beta])<
\nu([\beta])+\varepsilon\end{equation}
and
\begin{equation}\label{lower-e30}\langle \vec F,V_n(\sigma_\alpha,x)\rangle>
\langle \vec F,\nu\rangle-\vec\varepsilon.\end{equation} 

By \eqref{lower-e-13} and Lemma~\ref{lem-half}, for all $x\in\bigcup_{\omega\in E_n}[\omega]$ we have
\begin{equation}\label{lower-e-10} V_n(\sigma_\alpha,x)([\beta])<\frac{1}{2}+\varepsilon.\end{equation}
Since 
$V_n(\sigma_\alpha,x)$ may not belong to $\mathcal M(K_\alpha)$, we slightly reduce the cylinder sets spanned by elements of $E_n$ to decrease the relative frequency of the symbol $\beta$.
Let $\alpha_1^{\lfloor3\varepsilon n\rfloor}$ denote the $\lfloor3\varepsilon n\rfloor$-fold concatenation of $\alpha_1\in D_\alpha$ and
let \[E_n'=\{\omega\alpha_1^{\lfloor3\varepsilon n\rfloor}\colon\omega\in E_n\}.\] 
By \eqref{lower-e-10}, for any $x\in\bigcup_{\omega\in E_n'}[\omega]$ we have
\begin{equation}\label{lower-e50} V_{n+\lfloor3\varepsilon n\rfloor}(\sigma_\alpha,x)([\beta])<\frac{1}{2}\frac{(1+2\varepsilon)n}{n+\lfloor3\varepsilon n\rfloor}<
\frac{1}{2}.\end{equation}
For any $x\in\bigcup_{\omega\in E_n'}[\omega]$ and every $1\leq j\leq\ell$ we have 
\[\begin{split}\langle f_j,V_{n+\lfloor3\varepsilon n\rfloor}(\sigma_\alpha,x)\rangle=&\frac{1}{n+\lfloor3\varepsilon n\rfloor}\sum_{i=0}^{n+\lfloor3\varepsilon n\rfloor-1} f_j(\sigma_\alpha^{i}x)\\
=&\frac{1}{n}\sum_{i=0}^{n-1} f_j(\sigma_\alpha^{i}x)-\frac{\lfloor3\varepsilon n\rfloor}{n(n+\lfloor3\varepsilon n\rfloor)}\sum_{i=0}^{n-1} f_j(\sigma_\alpha^{i}x)\\
&+\frac{1}{n+\lfloor3\varepsilon n\rfloor}\sum_{i=n}^{n+\lfloor3\varepsilon n\rfloor-1} f_j(\sigma_\alpha^{i}x)\\
\geq& \frac{1}{n}\sum_{i=0}^{n-1} f_j(\sigma_\alpha^{i}x)
-6 \|f_j\|\varepsilon.\end{split}\]
For any $\omega\in E_n'$ and any $x\in[\omega]$ there exists 
$y(x)\in\Sigma_\alpha$ such that  $y(x)\in[x_0\cdots x_{n-1}]$. 
By Lemma~\ref{variation} and 
\eqref{lower-e10},
for all sufficiently large $n>1/(3\varepsilon)$ we have
\begin{equation}\label{lower-e40}\begin{split}\langle \vec F,V_{n+\lfloor3\varepsilon n\rfloor}(\sigma_\alpha,x)\rangle&\geq\langle \vec F,V_{n}(\sigma_\alpha,x)\rangle-6\|\vec F\|\cdot\vec\varepsilon\\
&>
\langle \vec F,V_{n}(\sigma_\alpha,y(x))\rangle-6\|\vec F\|\cdot\vec\varepsilon-\vec\varepsilon>\vec \Lambda.\end{split}\end{equation}

By \eqref{lower-e50} and Lemma~\ref{gyak-lem}, for all $x\in\bigcup_{\omega\in E_n'}[\omega]$ we have
 $V_{n+\lfloor3\varepsilon n\rfloor}(\sigma_\alpha,x)(K_\alpha)=1$, and so
 $V_{n+\lfloor3\varepsilon n\rfloor}(\sigma_\alpha,x)\in \mathcal G$
by \eqref{lower-e40}.
Moreover, for each $\omega\in E_{n}'$ the cylinder $[\omega]$ contains a point from $\phi_{\alpha}({\rm Per}_{\alpha0,n+\lfloor3\varepsilon n\rfloor}(\sigma))$.
Together with \eqref{lower-e20} 
we obtain 
  \[\begin{split}\#\left\{x\in \phi_\alpha({\rm Per}_{\alpha0,n+\lfloor3\varepsilon n\rfloor}(\sigma))\colon          V_{n+\lfloor3\varepsilon n\rfloor}(\sigma_\alpha,x)\in \mathcal G\right\}&\geq\#E_n'\\&\geq\exp(
 (h(\nu)-\varepsilon)n).\end{split}\]
Taking logarithms, dividing by $n+\lfloor3\varepsilon n\rfloor$ and then letting $n\to\infty$  yields
 \[\begin{split}\liminf_{n\to\infty}\frac{1}{n+\lfloor3\varepsilon n\rfloor}\log\widetilde\nu_{\alpha0,n+\lfloor3\varepsilon n\rfloor}( \mathcal G)&\geq\\
 \liminf_{n\to\infty}\frac{1}{n+\lfloor3\varepsilon n\rfloor}\log\#&\left\{x\in \phi_\alpha({\rm Per}_{\alpha0,n+\lfloor3\varepsilon n\rfloor}(\sigma))\colon          V_{n+\lfloor3\varepsilon n\rfloor}(\sigma_\alpha,x)\in \mathcal G\right\}\\
 &\hspace{0.5cm} -\limsup_{n\to\infty}\frac{1}{n}\log\#{\rm Per}_{\alpha0,n+\lfloor3\varepsilon n\rfloor}(\sigma)\\
 &\geq \frac{1}{1+3\varepsilon  }\cdot h(\nu)-\varepsilon-\log(M+1).\end{split}\]
The first inequality follows from the injectivity of $\phi_\alpha$. To estimate the exponential growth rate of $\#{\rm Per}_{\alpha0,n+\lfloor3\varepsilon n\rfloor}(\sigma)$ we have used 
 Lemmas~\ref{per-lem} and \ref{per-eq0}.
 For each sufficiently large $m\in\mathbb N$ there exists 
 $n(m)\in\mathbb N$ such that $n(m)+\lfloor3\varepsilon n(m)\rfloor=m$. Hence we obtain
 \[\begin{split}\liminf_{m\to\infty}\frac{1}{m}\log\widetilde\nu_{\alpha0,m }( \mathcal G)\geq \frac{1}{1+3\varepsilon  }\cdot h(\nu)-\varepsilon-\log (M+1).\end{split}\]
Finally, letting $\varepsilon\to0$ yields the desired inequality in Lemma~\ref{lowper0}.
  \end{proof}

\subsection{Intermediate upper bounds}\label{upper-sec}
The next two lemmas give large deviations upper bounds for particular closed sets.
 \begin{lemma}\label{uplem0}
Let $\ell\in\mathbb N$, $\vec{F}\in C(\Sigma_\alpha)^\ell$, $\vec \Lambda\in\mathbb R^\ell$
and let $\mathcal C$ be a non-empty closed 
subset of $\mathcal M(K_\alpha)$ of the form
\[\mathcal C=\{\nu\in \mathcal M(K_\alpha)\colon \langle\vec F, \nu\rangle\geq\vec \Lambda\}.\]
For any $\varepsilon>0$ there exists $n(\varepsilon)\geq1$ such that if $n\geq n(\varepsilon)$ and $\widetilde\nu_{\alpha0,n}(\mathcal C)>0$ then there exists $\nu_n\in \mathcal M(K_\alpha,\sigma_\alpha|K_\alpha)$ such that
\[\langle\vec F,\nu_n\rangle> \vec{\Lambda}-\vec{\varepsilon}
\ \text{ and } \ 
\frac{1}{n}\log\widetilde\nu_{\alpha0,n}(\mathcal C)\leq -J_\alpha(\nu_n)+\varepsilon.\]
\end{lemma}

\begin{lemma}\label{uplem0'}
Let $\ell\in\mathbb N$, $\vec{F}\in C(\Sigma_\beta)^\ell$, $\vec \Lambda\in\mathbb R^\ell$
and let $\mathcal C$ be a non-empty closed 
subset of $\mathcal M(K_\beta)$ of the form
\[\mathcal C=\{\nu\in \mathcal M(K_\beta)\colon \langle\vec F,\nu\rangle\geq\vec \Lambda\}.\]
For any $\varepsilon>0$ there exists $n(\varepsilon)\geq1$ such that if $n\geq n(\varepsilon)$ and $\widetilde\nu_{\beta,n}(\mathcal C)>0$ then there exists $\nu_n\in \mathcal M(K_\beta,\sigma_\beta|K_\beta)$ such that
\[\langle\vec F,\nu_n\rangle> \vec{\Lambda}-\vec{\varepsilon}
\ \text{ and } \ 
\frac{1}{n}\log\widetilde\nu_{\beta,n}(\mathcal C)\leq -J_\beta(\nu_n)+\varepsilon.\]
\end{lemma}

Below we only give a proof of Lemma~\ref{uplem0} since that of Lemma~\ref{uplem0'} is analogous.
\begin{proof}[Proof of Lemma~\ref{uplem0}]
Let $n\in\mathbb N$. For each $\omega\in L_n(\Sigma_\alpha)$
let $\overline{\omega}$
denote the element of ${\rm Per}_n(\sigma_\alpha)$ that is contained in $[\omega]$.
Let 
\[G_n=\{\omega\in L_n(\Sigma_\alpha) \colon\overline{\omega }\in \phi_\alpha({\rm Per}_{\alpha0,n}(\sigma))\text{ and } V_n(\sigma_\alpha,\overline{\omega})\in\mathcal C\}.\]

For each $n\in\mathbb N$ we set 
$\hat\sigma_\alpha=\sigma_\alpha^n$
and
 \[\Delta=\bigcap_{i=-\infty}^\infty  \hat\sigma_\alpha^{-i}\left(\bigcup_{\omega\in G_n}[\omega]\right).\]
The restriction of $\hat\sigma_\alpha$ to $\Delta$ is topologically 
 conjugate to the left shift on the full shift over the finite alphabet $G_n$ in the obvious way. Let $\hat\nu_n\in \mathcal M(\Delta,\hat\sigma_\alpha|\Delta)$ denote the measure of maximal entropy
 $\log\#G_n$.
The measure
$\nu_n = (1/n)\sum_{i=0}^{n-1}\hat\nu_n\circ \sigma_\alpha^{-i}$
belongs to $\mathcal M(\Sigma_\alpha,\sigma_\alpha)$, is ergodic
and satisfies 
\[h(\nu_n)n=
h(\hat\nu_n)=\log\#G_n.\]
For every $\omega\in G_n$ we have $\psi_\alpha(\overline\omega)\in{\rm Per}_{\alpha0,n}(\sigma)$, and so
$\sup_{x\in[\omega]} V_n(\sigma_\alpha,x)([\beta])\leq1/2$.
Since $\Delta$ is a closed set,
it follows that
\begin{equation}\label{follow1}\nu_n([\beta])\leq\frac{1}{2}.\end{equation}
By Lemma~\ref{variation}, for any $\varepsilon>0$ there exists $n(\varepsilon)\geq1$ such that if $n\geq n(\varepsilon)$ then 
 $\langle\vec F,V_{n}(\sigma_\alpha,x)\rangle\geq\vec \Lambda-\vec\varepsilon$ for all
$x\in\bigcup_{\omega\in G_n}[\omega]$.
This yields
\begin{equation}\label{follow2}\langle\vec F,\nu_{n}\rangle\geq\vec \Lambda-\vec\varepsilon.\end{equation}

Since the inequality in \eqref{follow1} may not be strict, $\nu_n$ may not give measure $1$ to $K_\alpha$.
 We construct another shift-invariant ergodic measure that gives measure $1$ to $K_\alpha$. 
Let 
$\alpha_1^\infty$ denote the fixed point of $\sigma_\alpha$ in $[\alpha_1]$.
Consider the convex combination
\[\nu_n'=(1-\varepsilon)\cdot\nu_n+\varepsilon\cdot\delta_{\alpha_1^\infty}\in \mathcal M(\Sigma_\alpha,\sigma_\alpha).\]
We have
\begin{equation}\label{erg-II''''}h(\nu_n')=(1-\varepsilon)\cdot h(\nu_n),\end{equation}
and by \eqref{follow1},
\begin{equation}\label{erg-II'''}\nu_n'([\beta])=(1-\varepsilon)\cdot\nu_n([\beta])<\frac{1}{2}.\end{equation}
and by \eqref{follow2}, \begin{equation}\label{erg-VI} \begin{split}
 \langle\vec F,\nu_n'\rangle&\geq(1-\varepsilon)\cdot\langle\vec F,\nu_n\rangle-\|\vec F\|\cdot\vec\varepsilon\\
 &\geq\langle \vec F,\nu_n\rangle-2\|\vec F\|\cdot\vec\varepsilon\geq\vec\Lambda-2\|\vec F\|\cdot\vec\varepsilon-\vec\varepsilon.\end{split}\end{equation}
By \eqref{erg-II''''}, \eqref{erg-II'''}, \eqref{erg-VI} and Lemma~\ref{e-dense-lem} applied to $\nu_n'$, there exists an ergodic measure $\nu_n''\in \mathcal M(\Sigma_\alpha,\sigma_\alpha)$ such that 
\begin{equation}\label{erg-III'}h(\nu_n'')\geq(1-\varepsilon)\cdot h(\nu_n)-\varepsilon,\end{equation}
 \begin{equation}\label{erg-IIII} \nu_n''([\beta])<\frac{1}{2} \ \text{ and } \
 \langle\vec F,\nu_n''\rangle>\vec\Lambda-2\|\vec F\|\cdot\vec\varepsilon-2\cdot\vec\varepsilon.\end{equation}
By the first inequality in \eqref{erg-IIII} and Lemma~\ref{gyak-lem}, $\nu_n''$ belongs to $\mathcal M(K_\alpha,\sigma_\alpha|K_\alpha)$.
Moreover we have
\[\begin{split}
\frac{1}{n}\log\widetilde\nu_{\alpha0,n }(\mathcal C)
=&\frac{1}{n }\log\#G_n-\frac{1}{n }\log\#\phi_\alpha({\rm Per}_{\alpha0,n }(\sigma))\\
=& h(\nu_n)-\frac{1}{n}\log\#{\rm Per}_{\alpha0,n}(\sigma)\\
\leq& \frac{1}{1-\varepsilon}\cdot h(\nu_n'')+\frac{\varepsilon}{1-\varepsilon} -\frac{1}{n}\log\#{\rm Per}_{\alpha0,n}(\sigma)\\
=& -J_{\alpha}(\nu_n'')+\frac{\varepsilon}{1-\varepsilon}\cdot h(\nu_n'')+\frac{\varepsilon}{1-\varepsilon}+\log(M+1)-\frac{1}{n}
\log\#{\rm Per}_{\alpha0,n}(\sigma)\\
\leq&  -J_{\alpha}(\nu_n'')+\frac{3\varepsilon}{1-\varepsilon}(\log(M+1)+1).
\end{split}\]
The second equality follows from
the injectivity of $\phi_\alpha$. To deduce
 the first inequality we have used \eqref{erg-III'}.
The last inequality holds for all sufficiently large $n\geq n(\varepsilon)$ by  
Lemmas~\ref{per-lem} and \ref{per-eq0}.
Since $\varepsilon>0$ is arbitrary, together with the second inequality in \eqref{erg-IIII} the proof of 
Lemma~\ref{uplem0} is complete.
\end{proof}

\subsection{Proof of Proposition~\ref{LDP-rest0}}\label{pfprop}
Let $\mathcal G$ be an open subset of $\mathcal M(K_\alpha)$.
By Lemma~\ref{top-lem}, sets of the form
$\{\nu\in \mathcal M(K_\alpha)\colon\langle\vec F, \nu\rangle>\vec \Lambda\}$ 
with $\ell\in\mathbb N$, $\vec{F}\in C(\Sigma_\alpha)^\ell$, $\vec \Lambda\in\mathbb R^\ell$
constitute a base of the weak* topology on $\mathcal M(K_\alpha)$.
Hence,
 $\mathcal G$ is written as the union 
$\mathcal G=\bigcup_{b\in B}\mathcal G_b$ of sets of this form. 
For each $\mathcal G_b$, Lemma~\ref{lowper0} gives
\[\liminf_{n\to\infty}\frac{1}{n}\log \widetilde\mu_{\alpha0,n}(\mathcal G_b)
\geq -\inf_{\mathcal G_b}J_\alpha,\]
and
hence 
\[\liminf_{n\to\infty}\frac{1}{n}\log\widetilde\mu_{\alpha0,n}(\mathcal G)
\geq\sup_{b\in B}(-\inf_{\mathcal G_b}J_\alpha)= -\inf_{\mathcal G} J_\alpha,
\]
as required in the lower bound in the LDP.

Let $\mathcal C'$ be a closed  
subset of $\mathcal M(K_\alpha)$.
From Lemma~\ref{top-lem},
there exists a closed subset $\mathcal C$ of $\mathcal M(\Sigma_\alpha)$ such that $\mathcal C'=\mathcal C\cap \mathcal M(K_\alpha).$
Let $\mathcal G$ be an arbitrary open subset of $\mathcal M(\Sigma_\alpha)$ that contains $\mathcal C$ and
put $\mathcal G'=\mathcal G\cap \mathcal M(K_\alpha).$
Since $\mathcal M(\Sigma_\alpha)$ is compact, metrizable and $\mathcal C$ is a compact subset of $\mathcal M(\Sigma_\alpha)$, we can choose $\varepsilon>0$ and finitely many closed subsets $\mathcal{C}_k$ $(k=1,\ldots,s)$ of $\mathcal M(\Sigma_\alpha)$ of the form 
$\mathcal C_k=\{\nu\in \mathcal M(\Sigma_\alpha)\colon \langle\vec F_k,\nu\rangle\geq\vec \Lambda_k\}$ 
with $\ell_k\in\mathbb N$, $\vec F_k\in C(\Sigma_\alpha)^{\ell_k}$, $\vec \Lambda_k\in\mathbb R^{\ell_k}$
such that 
\[\mathcal C \subset \bigcup_{k=1}^s \mathcal C_k \subset 
\bigcup_{k=1}^s \mathcal C_k(\varepsilon)\subset \mathcal G,
\]
where
$\mathcal C_k(\varepsilon)=\{\nu\in\mathcal M(\Sigma_\alpha)\colon \langle\vec F_k,\nu\rangle>\vec{\Lambda}_k-\vec{\varepsilon}\}.$
Put $\mathcal C_k'=\mathcal C_k\cap \mathcal M(K_\alpha)$ and $\mathcal C_k(\varepsilon)'=\mathcal C_k(\varepsilon)\cap \mathcal M(K_\alpha)$ for $1\leq k\leq s$.
We have
\[\mathcal C' \subset \bigcup_{k=1}^s \mathcal C_k' \subset 
\bigcup_{k=1}^s \mathcal C_k(\varepsilon)'\subset \mathcal G'.\]
From Lemma~\ref{uplem0} applied to each set
$\mathcal C_k'=\{\nu\in \mathcal M(K_\alpha)\colon \langle\vec F_k, \nu\rangle\geq\vec \Lambda_k\}$, we have
\[\begin{split}
\limsup_{n\to\infty}\frac{1}{n}\log\widetilde\nu_{\alpha0,n}(\mathcal C')&\leq\max_{1\leq k\leq s}\limsup_{n\to\infty}\frac{1}{n}\log\widetilde\nu_{\alpha0,n}(\mathcal C_k')\\
& \le \max_{1\le k\le s}  
\left(-\inf_{\mathcal C_k(\varepsilon)'}J_\alpha\right)+\varepsilon \le
-\inf_{\mathcal G'}J_\alpha+\varepsilon.\end{split}
\]
Since $\varepsilon>0$ is arbitrary and
$\mathcal G'$ ranges over any open set containing 
 $\mathcal C'$,
 we obtain
\[
\limsup_{n\to\infty}\frac{1}{n}\log\widetilde\nu_{\alpha0,n}(\mathcal C')\leq
\inf_{\mathcal G' \supset \mathcal C'} \left(-\inf_{\mathcal G'} J_\alpha\right)=
-\inf_{\mathcal C'}J_\alpha\]
as required. The last equality is due to the lower semicontinuity of $J_\alpha$ by Lemma~\ref{J-rate} below.
We have verified the LDP for
$(\widetilde\nu_{\alpha0,n})_{n=1}^\infty$ with the rate function $J_\alpha$.
The same argument with Lemmas~\ref{lowper0'} and \ref{uplem0'} yields the LDP for
$(\widetilde\nu_{\beta,n})_{n=1}^\infty$ with the rate function $J_\beta$.

\begin{lemma}\label{J-rate}For each $\gamma\in\{\alpha,\beta\}$, $J_\gamma$ is lower semicontinuous.\end{lemma}
\begin{proof}
Let $\nu\in\mathcal M(K_\gamma)$.
To show the lower semicontinuity of $J_\gamma$ at $\nu$,
let $(\nu_n)_{n=1}^\infty$ be a sequence that converges to $\nu$ in the weak* topology on $\mathcal M(K_\gamma)$. If $\nu_n\notin \mathcal M(K_\gamma,\sigma_\gamma|K_\gamma)$
for all sufficiently large $n$,
then clearly we have $J_\gamma(\nu)\leq \infty=\liminf_{n\to\infty}J_\gamma(\nu_n)$. 
Otherwise, taking a subsequence if necessary we may assume $\nu_n\in \mathcal M(K_\gamma,\sigma_\gamma|K_\gamma)$ for every $n\geq1$.
As in Remark~\ref{view-rem}, we may identify elements of $\mathcal M(K_\gamma,\sigma_\gamma|K_\gamma)$ as elements of $\mathcal M(\Sigma_\gamma,\sigma_\gamma)$.
 By Lemma~\ref{top-lem}, 
$(\nu_n)_{n=1}^\infty$ converges to $\nu$ in the weak* topology on $\mathcal M(\Sigma_\gamma)$.
Since $\mathcal M(\Sigma_\gamma,\sigma_\gamma)$ is a closed subset of $\mathcal M(\Sigma_\gamma)$, we obtain  $\nu\in\mathcal M(\Sigma_\gamma,\sigma_\gamma)$.
Since the entropy function is upper semicontinuous on $\mathcal M(\Sigma_\gamma,\sigma_\gamma)$, we obtain $J_\gamma(\nu)\leq \liminf_{n\to\infty}J_\gamma(\nu_n)$.
\end{proof}

\begin{lemma}\label{closed}
For each $\gamma\in\{\alpha,\beta\}$,
$\mathcal M_{\gamma0}(\Sigma_D,\sigma)$ is a closed subset of $\mathcal M(\Sigma_D)$.
\end{lemma}
\begin{proof} 
Let $\mu\in \mathcal M(\Sigma_D,\sigma)\setminus \mathcal M_{\alpha0}(\Sigma_D,\sigma)$. By Lemma~\ref{trichotomy} and the ergodic decomposition theorem, we have $\mu(A_\beta)>0$. The ergodic decomposition theorem and Birkhoff's ergodic theorem together imply $\mu([D_\beta]
)>0$. Since the indicator function of the union $[D_\beta]$ of the cylinder sets is continuous, there exists an open subset $\mathcal G$ of $\mathcal M(\Sigma_D)$
containing $\mu$ such that for all $\mu'\in \mathcal G\cap\mathcal M(\Sigma_D,\sigma)$ we have $\mu'([D_\beta])>0$, and so  $\mu'\in\mathcal M_{\alpha0}(\Sigma_D,\sigma)$.  Hence, $\mathcal M_{\alpha0}(\Sigma_D,\sigma)$ is a closed subset of $\mathcal M(\Sigma_D)$. The same argument yields the closedness of $\mathcal M_{\beta0}(\Sigma_D,\sigma)$ in $\mathcal M(\Sigma_D)$. \end{proof}

It is left to show that $J_\alpha$ and $J_\beta$ are good rate functions. 
Let $c>0$.
Let $(\nu_n)_{n=1}^\infty$ be an arbitrary sequence in $\{\nu\in \mathcal M(K_\alpha)\colon J_\alpha(\nu)\leq c\}$. 
By $A_\alpha\cup A_0\subset B_\alpha$ and $A_\beta\cap B_\beta=\emptyset$ as in \eqref{subset-AB},
 $\nu_n\circ\psi_\alpha^{-1}$ belongs to $\mathcal M_{\alpha0}(\Sigma_D,\sigma)$ for every $n\geq1$. By Lemma~\ref{closed}, 
 $\mathcal M_{\alpha0}(\Sigma_D,\sigma)$ is closed. Hence,
there is a subsequence $(\nu_{n(k)}\circ\psi_\alpha^{-1})_{k=1}^\infty$ that converges to $\mu\in\mathcal M_{\alpha0}(\Sigma_D,\sigma)$ in the weak* topology on $\mathcal M(\Sigma_D)$.
The upper semicontinuity of the entropy function gives  $h(\mu)\geq\limsup_{k\to\infty}
h(\nu_{n(k)}\circ\psi_\alpha^{-1})$. 
By Lemma~\ref{top-lem}, 
$(\nu_{n(k)}\circ\psi_\alpha^{-1})_{k=1}^\infty$ converges to $\mu$ in the weak* topology on $\mathcal M(A_\alpha\cup A_0)$.
Since $\phi_\alpha$ is continuous
and commutes with the left shifts,
 $(\nu_{n(k)})_{k=1}^\infty$
converges to $\mu\circ\phi_\alpha^{-1}\in\mathcal M(K_\alpha)$. The lower semicontinuity of $J_\alpha$ in Lemma~\ref{J-rate} yields
$J_\alpha(\mu\circ\phi_\alpha^{-1})\leq \liminf_{k\to\infty}J_\alpha(\nu_{n(k)})\leq c$.
We have verified the compactness of $\{\nu\in \mathcal M(K_\alpha)\colon J_\alpha(\nu)\leq c\}$. Since $c>0$ is arbitrary, $J_\alpha$ is a good rate function. The same argument shows that $J_\beta$ is a good rate function too. This completes the proof of Proposition~\ref{LDP-rest0}. \qed

\medskip

\subsection{Proof of Theorem~\ref{theorema}}\label{pfthmc}
For each $\gamma\in\{\alpha,\beta\}$, consider the continuous map
$\nu\in\mathcal M(K_\gamma)\mapsto \nu\circ\psi_\gamma^{-1}\in \mathcal M(\Sigma_D)$. These two maps induce the push-forwards from $\mathcal M(\mathcal M(K_\alpha))$ to $\mathcal M(\mathcal M(\Sigma_D))$ and $\mathcal M(\mathcal M(K_\beta))$ to $\mathcal M(\mathcal M(\Sigma_D))$ that send $\widetilde\nu_{\alpha0,n}$
to $\widetilde\mu_{\alpha0,n}$
and $\widetilde\nu_{\beta,n}$ to
$\widetilde\mu_{\beta,n}$ 
 respectively. 
\begin{lemma}[Contraction Principle \cite{DemZei98}] \label{CP}Let $\mathcal X$, $\mathcal Y$ be Hausdorff spaces and let $(\mu_n)_{n=1}^\infty$ be a sequence of Borel probability measures on $\mathcal X$ for which the LDP holds with a good rate function $I$. Let $f\colon \mathcal X\to \mathcal Y$ be a continuous map. Then the LDP holds for $(\mu_n\circ f^{-1})_{n=1}^\infty$ with a good rate function $J\colon \mathcal Y\to [0,\infty]$ given by  \[J(y)=\inf\{I(x)\colon x\in \mathcal X,\ f(x)=y\}.\]\end{lemma}
By Proposition~\ref{LDP-rest0}, the LDP holds for $(\widetilde\nu_{\alpha0,n})_{n=1}^\infty$ and 
$(\widetilde\nu_{\beta,n})_{n=1}^\infty$ with the good rate functions $J_\alpha$ and $J_\beta$ respectively.
 By Lemma~\ref{CP}, the LDP holds for  $(\widetilde\mu_{\alpha0,n})_{n=1}^\infty$ and  $(\widetilde\mu_{\beta,n})_{n=1}^\infty$ with the rate functions $ I_\alpha\colon \mathcal M(\Sigma_D)\to [0,\infty]$ and $ I_\beta\colon \mathcal M(\Sigma_D)\to [0,\infty]$ respectively given 
by 
\[I_\gamma(\mu)=\inf\{J_\gamma(\nu)\colon \nu\in \mathcal M(K_\gamma),\ \nu\circ\psi_{\gamma}^{-1}=\mu\}.\]
The definition of $J_\gamma$ in Proposition~\ref{LDP-rest0} 
gives
$I_\gamma(\mu)=\log(M+1)-h(\mu)$ if $\mu\in \mathcal M_{\gamma0}(\Sigma_D,\sigma)$ and
$I_\gamma(\mu)=\infty$ otherwise. The proof of Theorem~\ref{theorema} is complete.
    \qed

\subsection{Proof of Theorems~\ref{break-Dyck} and \ref{rate-Dyck}}\label{pf-final}
For each $n\in\mathbb N$, notice that \begin{equation}\label{mu-tik}
\widetilde\mu_n=\frac{\#({\rm Per}_{\alpha0,n}(\sigma) ) }{\#{\rm Per}_{n}(\sigma)}\cdot\widetilde\mu_{\alpha0,n}+\frac{\#{\rm Per}_{\beta,n}(\sigma) }{\#{\rm Per}_{n}(\sigma)}\cdot\widetilde\mu_{\beta,n}.
\end{equation}
By Theorem~\ref{theorema}, the LDP holds for $(\widetilde\mu_{\alpha0,n})_{n=1}^\infty$ and $(\widetilde\mu_{\beta,n})_{n=1}^\infty$ with the rate functions $I_\alpha$ and $I_\beta$ respectively.
For any open subset $\mathcal G$ of $\mathcal M(\Sigma_D)$, we have
\[\begin{split}\liminf_{n\to\infty}\frac{1}{n}\log\widetilde\mu_n(\mathcal G)
&\geq\max\left\{\liminf_{n\to\infty}\frac{1}{n}\log\widetilde\mu_{\alpha0,n}(\mathcal G),\liminf_{n\to\infty}\frac{1}{n}\log\widetilde\mu_{\beta,n}(\mathcal G)\right\}\\
&\geq\max\left\{-\inf_{\mathcal G} I_\alpha,-\inf_{\mathcal G} I_\beta\right\}= -\inf_{\mathcal G} I.\end{split}\]
The first inequality is because 
the two coefficients in the decomposition \eqref{mu-tik} stay bounded away from $0$ as $n\to\infty$ by Lemmas~\ref{per-lem} and \ref{per-eq0}.

Similarly, for any closed subset
 $\mathcal C$ of $\mathcal M(\Sigma_D)$ we have
\[\begin{split}\limsup_{n\to\infty}\frac{1}{n}\log\widetilde\mu_n(\mathcal C)
&\leq\max\left\{\limsup_{n\to\infty}\frac{1}{n}\log\widetilde\mu_{\alpha0,n}(\mathcal C),\limsup_{n\to\infty}\frac{1}{n}\log\widetilde\mu_{\beta,n}(\mathcal C)\right\}\\
&\leq\max\left\{-\inf_{\mathcal C} I_\alpha,-\inf_{\mathcal C }  I_\beta\right\}= -\inf_{\mathcal C}I.\end{split}\]
We have verified the LDP for $(\widetilde\mu_n)_{n=1}^\infty$ with the rate function $I$.
Clearly $I$ is not convex: if $\mu_1\in\mathcal M_{\alpha0}(\Sigma_D,\sigma)\setminus\mathcal M_{\beta0}(\Sigma_D,\sigma)$, 
$\mu_2\in\mathcal M_{\beta0}(\Sigma_D,\sigma)\setminus\mathcal M_{\alpha0}(\Sigma_D,\sigma)$ and
 $0<t<1$ then $(1-t)I(\mu_1)+tI(\mu_2)<\infty=I((1-t)\cdot\mu_1+t\cdot\mu_2)$.
 The proof of Theorems~\ref{break-Dyck} and \ref{rate-Dyck} is complete.
\qed

\subsection{Proof of Theorem~\ref{level1-thm}}\label{last-sec}
Let $f\colon\Sigma_D\to\mathbb R$ be continuous.
The desired superposition form of $I_f$ is a consequence of the Contraction Principle in 
Lemma~\ref{CP} applied to the level-2 LDP in
Theorem~\ref{break-Dyck} with the rate function given in Theorem~\ref{rate-Dyck}. Below we prove (a) (b).

Let $\gamma\in\{\alpha,\beta\}$.
Since $f$ is fixed throughout, for simplicity we write $L_\gamma$ for $L_{\gamma,f}$.
From the affinity of $\mu\in\mathcal M_{\gamma0}(\Sigma_D,\sigma)\mapsto \langle f,\mu\rangle$ and the entropy function on 
$\mathcal M_{\gamma0}(\Sigma_D,\sigma)$, $I_{f,\gamma}$ is convex on $L_\gamma$. Hence, $I_{f,\gamma}$ is continuous on the interior of $L_\gamma$. From the convexity and the lower semicontinuity, $I_{f,\gamma}$ is continuous at the boundary points of $L_\gamma$. 
The proof of (a) is complete.


 \medskip

To prove (b), let $f\in\mathscr{F}$.
For each $\gamma\in\{\alpha,\beta\}$
let $\pi_\gamma\colon x\in\Sigma_\gamma\mapsto\pi_\gamma(x)\in\{\alpha,\beta\}^\mathbb Z$ denote the natural projection: $\pi_\alpha(x)$ (resp. $\pi_\beta(x)$) is obtained by replacing all $\alpha_k$, $1\leq k\leq M$ (resp. all $\beta_k$, $1\leq k\leq M$) in $x$ by $\alpha$ (resp. by $\beta$).
Since $f\in\mathscr{F}$,
there exists a H\"older continuous function $\bar f\colon\{\alpha,\beta\}^\mathbb Z\to\mathbb R$ with $f=\bar f\circ\pi_D$. Then
$f_\gamma=\bar f\circ\pi_\gamma$ is H\"older continuous and $f_\gamma=f\circ\psi_\gamma$.
We now introduce a {\it multifractal spectrum}
 $t\in L_\gamma\mapsto b_\gamma(t)\in\mathbb R$ by
\begin{equation}\label{spectrum}b_\gamma(t)=\sup\left\{\frac{h(\nu)}{t+c_0}\colon\nu\in\mathcal M(\Sigma_\gamma,\sigma_\gamma),\ \langle f_\gamma,\nu\rangle=t\right\},\end{equation}
where $c_0>-\min f$ is a fixed constant. 
By the definition of $L_\gamma$, the range of this supremum is a non-empty set.

The domain of the multifractal spectrum is characterized as follows.
Put \[t_\gamma^-=\inf_{x\in B_\gamma}\liminf_{n\to\infty} \frac{1}{n}S_nf(x)\ \text{ and }\ t_\gamma^+=\sup_{x\in B_\gamma}\limsup_{n\to\infty} \frac{1}{n}S_nf(x).\] 
\begin{lemma}\label{comp-sub}For each $\gamma\in\{\alpha,\beta\}$ we have $L_\gamma=[t_\gamma^-,t_\gamma^+]$.\end{lemma} \begin{proof}From the ergodic decomposition theorem, there exists an ergodic measure $\mu_0\in\mathcal M_{\gamma0}(\Sigma_D,\sigma)$ such that $\min L_\gamma=\langle f,\mu_0\rangle$. By Birkhoff's ergodic theorem,   there exists $x\in B_\gamma$ such that $\lim_{n\to\infty}(1/n)S_nf(x)=\langle f,\mu_0\rangle$. This yields $\min L_\gamma\geq t_\gamma^-$. For any $\varepsilon>0$ there exists $x\in B_\gamma$ such that $\liminf_{n\to\infty} (1/n)S_nf(x)<t_\gamma^-+\varepsilon/2$. If $\min L_\gamma> t_\gamma^-$, then by Lemma~\ref{variation} there exists $n\in\mathbb N$ such that  \begin{equation}\label{comp-sub1}\sup_{y\in[x_0\cdots x_{n-1}] }\frac{1}{n}(S_nf(x)-S_nf(y))\leq\min L_\gamma-t_\gamma^--\varepsilon\end{equation} and 
\begin{equation}\label{comp-sub2}\frac{1}{n}S_nf(x)<t_\gamma^-+\varepsilon.\end{equation}  Let $z$ denote the element of ${\rm Per}_n(\sigma)$ that is contained in $[x_0\cdots x_{n-1}]$. Applying \eqref{comp-sub1} with $y=z$ and then \eqref{comp-sub2}, \[\langle f,V_n(\sigma,z)\rangle=\frac{1}{n}S_nf(z)\leq\min L_\gamma-t_\gamma^--\varepsilon+\frac{1}{n}S_n(x)<\min L_\gamma,\] which yields a contradiction to the definition of $L_\gamma$. Hence we obtain $\min L_\gamma= t_\gamma^-$. An analogous argument shows $\max L_\gamma= t_\gamma^+$. Since $L_\gamma$ is connected, $L_\gamma=[t_\gamma^-,t_\gamma^+]$ holds.\end{proof}
\begin{lemma}\label{lem-analytic}For each $\gamma\in\{\alpha,\beta\}$, 
$t\in(t_\gamma^-,t_\gamma^+)\mapsto b_\gamma(t)\in\mathbb R$ is analytic. \end{lemma}
\begin{proof}Define a {\it pressure function} $P\colon s\in\mathbb R\mapsto P(s)\in\mathbb R$ by
\[P(s)=\sup\left\{h(\nu)+s\cdot\langle f_\gamma+c_0,\nu\rangle\colon\nu\in\mathcal M(\Sigma_\gamma,\sigma_\gamma)\right\}.\]
Then $P$ is analytic, and  
the supremum is attained by a unique measure, denoted by $\nu_{s}$, which satisfies
$P'(s)=\langle f_\gamma+c_0, \nu_{s}\rangle>0$
(see \cite{Rue04}).
  We have
\[t_\gamma^-=\lim_{s\searrow-\infty}P'(s)\ \text{ and }\
t_\gamma^+=\lim_{s\nearrow +\infty}P'(s).\]
 By the implicit function theorem, there exists a strictly increasing analytic function $s\colon t\in(t_\gamma^-,t_\gamma^+)\mapsto s(t)\in\mathbb R $ satisfying 
$P'(s(t))=t+c_0$.
Let $t\in (t_\gamma^-,t_\gamma^+)$. We have
\[\begin{split}P(s(t))-(t+c_0)s(t)&=
h(\nu_{s(t)})+s(t)\cdot\langle f_\gamma+c_0,\nu_{s(t)}\rangle-s(t)(t+c_0)\\
&=
h(\nu_{s(t)})+s(t)(t+c_0)-s(t)(t+c_0)\\
&=h(\nu_{s(t)})\leq(t+c_0) b_\gamma(t).\end{split}\]
Since $f_\gamma$ is continuous and the entropy function on $\mathcal M(\Sigma_\gamma,\sigma_\gamma)$ is upper semicontinuous, the supremum in the definition of $b_\gamma(t)$ is attained, say by
 $\nu_*$. 
Then \[\begin{split}(t+c_0) b_\gamma(t)=h(\nu_*)&= h(\nu_*)+(t+c_0)s(t)-(t+c_0)s(t)\\
&= h(\nu_*)+s(t)\cdot\langle f_\gamma+c_0,\nu_*\rangle-(t+c_0)s(t)\\
&\leq P(s(t))-(t+c_0)s(t).\end{split}\] 
Combining these two inequalities yields
$
b_\gamma(t)=P(s(t))/(t+c_0)-s(t).$
The analyticity of $b_\gamma(t)$ follows from that of
 $P(s)$ and $s(t)$.
\end{proof}


To relate the multifractal spectra to $I_{f,\alpha}$ and $I_{f,\beta}$, define
\[U_\alpha=\left\{t\in L_\alpha\colon b_\alpha(t)=\sup\left\{\frac{h(\nu)}{t+c_0}\colon\nu\in\mathcal M(\Sigma_\alpha,\sigma_\alpha),\ \langle f_\alpha,\nu\rangle=t,\ \nu([\beta])<\frac{1}{2}\right\}\right\}\]
and
\[U_\beta=\left\{t\in L_\beta\colon b_\beta(t)=\sup\left\{\frac{h(\nu)}{t+c_0}\colon\nu\in\mathcal M(\Sigma_\beta,\sigma_\beta),\ \langle f_\beta,\nu\rangle=t,\ \nu([\alpha])<\frac{1}{2}\right\}\right\}\]
\begin{lemma}\label{open-lem}For each $\gamma\in\{\alpha,\beta\}$, $U_\gamma$ contains a neighborhood of $\langle f,\mu_\gamma\rangle$ in $L_\gamma$.\end{lemma}
\begin{proof} 
Since $\nu_\alpha$ is a maximal entropy measure for $\sigma_\alpha$
and satisfies $\nu_\alpha([\beta])<1/2$, we have $\langle f_\alpha,\nu_\alpha\rangle\in U_\alpha$.
The relations $\mu_\alpha=\nu_\alpha\circ\psi_\alpha^{-1}$ and $f\circ\psi_\alpha= f_\alpha$ give
  $\langle f_\alpha,\nu_\alpha\rangle=\langle f,\mu_\alpha\rangle$. Hence we obtain
  $\langle f,\mu_\alpha\rangle\in U_\alpha$.
If there exists a sequence $(t_n)_{n=1}^\infty$ in $L_\alpha\setminus U_\alpha$ converging to
$\langle f,\mu_\alpha\rangle$, then
there exists a sequence $(\nu_n)_{n=1}^\infty$ in
$\mathcal M(\Sigma_\alpha,\sigma_\alpha)$ such that $\langle f_\alpha,\nu_n\rangle=t_n$
and $\nu_n([\beta])\geq1/2$
for every $n\geq1$, 
and $h(\nu_n)/(t+c_0)\to b_{\alpha}(\langle f_\alpha,\nu_\alpha\rangle)$ as $n\to\infty$. 
Let $\nu$ be an accumulation point of $(\nu_n)_{n=1}^\infty$ in the weak* topology.
Then $\nu([\beta])\geq1/2$ holds, and the upper semicontinuity of entropy function on $\mathcal M(\Sigma_\alpha,\sigma_\alpha)$ gives $h(\nu)\geq h(\nu_\alpha)$. Since $\nu_\alpha$ is the unique maximal entropy measure for $\sigma_\alpha$ we obtain $\nu=\nu_\alpha$, a contradiction to $\nu_\alpha([\beta])<1/2$. A proof for the case $\gamma=\beta$ is analogous.
\end{proof}

Let ${\rm int}(U_\gamma)$ denote the interior of $U_\gamma$ relative to $\mathbb R$.
\begin{lemma}\label{U-lemma}
For each $\gamma\in\{\alpha,\beta\}$,
we have
$I_{f,\gamma}(t)=\log(M+1)-(t+c_0) b_{\gamma}(t)$ for all $t\in{\rm int}(U_\gamma)$.\end{lemma}
\begin{proof}
We claim that for all $t\in U_\alpha$,
\begin{equation}\label{claim-eq1}\begin{split}\sup&\{h(\mu)\colon\mu\in\mathcal M_{\alpha0} (\Sigma_D,\sigma),\ \langle f,\mu\rangle=t\}\\
&\geq
\sup\left\{h(\nu)\colon\nu\in\mathcal M(\Sigma_\alpha,\sigma_\alpha),\ \text{ergodic,}\ \langle f_\alpha,\nu\rangle=t,\ \nu([\beta])<\frac{1}{2}\right\}\\
&=
\sup\left\{h(\nu)\colon\nu\in\mathcal M(\Sigma_\alpha,\sigma_\alpha),\  \langle f_\alpha,\nu\rangle=t,\ \nu([\beta])<\frac{1}{2}\right\}\\
&=
\sup\left\{h(\nu)\colon\nu\in\mathcal M(\Sigma_\alpha,\sigma_\alpha),\  \langle f_\alpha,\nu\rangle=t\right\}.\end{split}\end{equation}
The inequality in \eqref{claim-eq1} follows from $f\circ\psi_\alpha=f_\alpha$ and Lemma~\ref{gyak-lem}, and the first equality follows from Lemma~\ref{e-dense-lem}. 
The last inequality is because $t\in U_\alpha$.
 For any $\mu\in\mathcal M_{\alpha0} (\Sigma_D,\sigma)$ with $\langle f,\mu\rangle=t$ we have
$\langle f_\alpha,\mu\circ\phi_\alpha^{-1}\rangle=t$, and
Lemma~\ref{lem-half} gives $\mu\circ\phi_\alpha^{-1}([\beta])\leq 1/2$.
Let $\alpha_1^\infty$ denote the fixed point of $\sigma_\alpha$ in the cylinder set $[\alpha_1]$, and consider the convex combination
$\xi_\varepsilon=(1-\varepsilon)\cdot\mu\circ\phi_\alpha^{-1}+\varepsilon\cdot\delta_{\alpha_1^\infty}$
for $\varepsilon\in(0,1)$.
We have $\xi_\varepsilon([\beta])<1/2$ and $\langle f_\alpha,\xi_\varepsilon\rangle\to t$, $h(\xi_\varepsilon)\to h(\mu)$
as $\varepsilon\to0$. Since the last supremum in \eqref{claim-eq1} is a continuous function of $t$ on the interior of $U_\alpha$ by Lemmas~\ref{comp-sub} and \ref{lem-analytic}, the inequality in \eqref{claim-eq1} must be the equality provided $t\in{\rm int}(U_\gamma)$.

Then, for $t\in{\rm int}(U_\gamma)$ we have 
\[\begin{split}I_{f,\alpha}(t)=&\inf\{I(\mu)\colon\mu\in\mathcal M_{\alpha0}(\Sigma_D,\sigma),\ \langle f,\mu\rangle=t\}\\
=&\log(M+1)-\sup\{h(\mu)\colon\mu\in\mathcal M_{\alpha0} (\Sigma_D,\sigma),\ \langle f,\mu\rangle=t\}\\
=&\log(M+1)-(t+c_0)b_{\alpha}(t),\end{split}\] 
as required.
A proof for the case $\gamma=\beta$ is analogous.\end{proof}

  All the assertions in
(b) except the strict convexity follow from Lemmas~\ref{open-lem} and \ref{U-lemma}. 
To prove the strict convexity we appeal to the next lemma. In the lemma and the proof below, the $'$ denotes the derivative.
\begin{lemma}\label{convex-lem}
Let $W\subset\mathbb R$ be an open interval and let $g\colon W\to\mathbb R$
be a monotone, convex  analytic function. If $g$ is not affine, then $g''>0$ on $W$.\end{lemma}
\begin{proof}
Let $x\in W$.
Assume $g''(x)=0$. By the convexity of $g$, $g''\geq0$ and so $g'''(x)=0$.
If $g''''(x)\neq0$, then $x$ is a local maximal or minimal point of $g$, a contradiction. So, 
$g''''(x)=0$. If $g'''''(x)=(g'')'''(x)\neq0$, then $g''$ becomes negative near $x$, a contradiction.
So, $g'''''(x)=0$.
If $g^{(6)}(x)\neq0$,
then $x$ is a local minimal or maximal point of $g$, a contradiction. So 
$g^{(6)}(x)=0$.
In this way, we get $g^{(n)}(x)=0$ for all $n\geq2$.
By its analyticity,
 $g$ must be an affine function, a contradiction.
\end{proof}

Applying Lemma~\ref{convex-lem} to $I_{f,\gamma}$ that has been shown to be convex analytic, we obtain the strict convexity in (b). The proof of Theorem~\ref{level1-thm} is complete. \qed



\subsection{Proof of \eqref{ratef-1}}\label{ratepf}
For $f$ in \eqref{phi-0} we have $L_{f,\gamma}=[0,1]$ for $\gamma\in\{\alpha,\beta\}$.
Let us consier the two multifractal spectra $b_\alpha\colon [0,1]\to\mathbb R$, $b_\beta\colon [0,1]\to\mathbb R$ by \eqref{spectrum}
with $c_0=1$, where
 $f_\alpha\colon\Sigma_\alpha\to\mathbb R$ is given by $f_\alpha(x)=0$ if $x_0\in D_\alpha$ and $f_\alpha(x)=1$ if $x_0=\beta$, and 
$f_\beta\colon\Sigma_\beta\to\mathbb R$ is given by $f_\beta(x)=0$ if $x_0=\alpha$ and $f_\beta(x)=1$ if $x_0\in D_\beta$. 
Direct calculations show
$U_\alpha=[0,1/2)$ and $U_\beta=(1/2,1]$.

By \cite{Kif90,OrePel89,Tak84} (or by Sanov's theorem \cite{DemZei98}), for each $\gamma\in\{\alpha,\beta\}$
the level-2 LDP holds for the sequence $(\nu_\gamma\{x\in\Sigma_\gamma\colon V_n(\sigma_\gamma,x)\in\cdot\})_{n=1}^\infty$
of distributions of empirical measures, with the rate function $Q_\gamma\colon\mathcal M(\Sigma_\gamma)\to[0,\infty]$ given by \[Q_\gamma(\nu)=\begin{cases}\log(M+1)-h(\nu)&\text{ if }\nu\in \mathcal M(\Sigma_\gamma,\sigma_\gamma),\\
    \infty&\text{ otherwise}.\end{cases}\]
By the Contraction Principle, the level-1 LDP holds
for $(\nu_\gamma\{x\in\Sigma_\gamma\colon(1/n)S_nf_\gamma(x)\in\cdot\})_{n=1}^\infty$ with the rate function $R_\gamma\colon\mathbb R\to[0,\infty]$ given by \[R_{\gamma}(t)=\inf\{Q_\gamma(\nu)\colon\nu\in\mathcal M(\Sigma_\gamma),\ \langle f_\gamma,\nu\rangle=t\}.\]
Notice that
$R_\gamma(t)=\log(M+1)-(t+1)b_\gamma(t)$ for all $t\in [0,1]$.

Since $(f_\gamma\circ\sigma_\gamma^n)_{n=1}^\infty$ is a sequence of independently and identically distributed random variables with respect to $\nu_\gamma$, 
the rate function $R_\gamma$ can be computed directly (see e.g., \cite[Lemma~I.3.2]{Ell85}):\[R_\alpha(t)=\begin{cases}t\log t+(1-t)\log(1-t)+\log((M+1)/M^{1-t})  & \text{ if }t\in[0,1],\\
\infty & \text{ if }t\in\mathbb R\setminus[0,1]\end{cases}\]
and
\[R_\beta(t)=\begin{cases}t\log t+(1-t)\log(1-t)+\log((M+1)/M^t)  & \text{ if }[0,1],\\
\infty & \text{ if }\mathbb R\setminus[0,1].\end{cases}\]
By Lemma~\ref{gyak-lem}, we have  
 $I_{f,\alpha}(t)=R_\alpha(t)$ if $t\in(0,1/2)$ and $I_{f,\beta}(t)=R_\beta(t)$ if $t\in(1/2,1)$.
 The continuity in 
 Theorem~\ref{level1-thm}(a) yields $I_{f,\alpha}(0)=R_\alpha(0)$
 and $I_{f,\beta}(1)=R_\beta(1)$.
From \cite[Remark~2.1]{Mey08} and the formula for $I_f$ in Theorem~\ref{rate-Dyck}, we obtain
$I_f(1/2)=\log((M+1)/\sqrt{4M})$
which equals $R_\alpha(1/2)$.
The proof of Proposition~\ref{ratef-1} is complete.\qed

\appendix \def\thesection{\Alph{section}}
\section{Non-existence of continuous extension}
In the proofs of Lemmas \ref{lowper0} and \ref{uplem0} (resp. Lemmas \ref{lowper0'} and \ref{uplem0'}), we have approximated in the weak* topology the elements of $\mathcal M(\Sigma_\alpha)$ (resp. $\mathcal M(\Sigma_\beta)$) arising in the constructions by elements of $\mathcal M(K_\alpha)$ (resp. $\mathcal M(\Sigma_\beta)$). A glance at the proofs shows that
these approximations would not be necessary if the map $\psi_\alpha\colon K_\alpha\to \Sigma_D$ (resp. $\psi_\beta\colon K_\beta\to \Sigma_D$) had a continuous extension to the whole shift space $\Sigma_\alpha$ (resp. $\Sigma_\beta$). 
Actually 
 this is not the case.

\begin{lemma}\label{no-extension}For each $\gamma\in\{\alpha,\beta\}$, 
 there is no continuous map $\bar\psi_\gamma\colon\Sigma_\gamma\to\Sigma_D$ such that $\bar\psi_\gamma(x)=\psi_\gamma(x)$ for all $x\in K_\gamma$.\end{lemma} \begin{proof}
 We define $\bar\phi_\gamma\colon \Sigma_D\to \Sigma_\gamma$ as follows: if $\gamma=\alpha$ (resp. $\gamma=\beta$), then define $\bar\phi_\gamma(x)$ for $x\in\Sigma_D$ by replacing all $\beta_k$ (resp. $\alpha_k$), $1\leq k\leq M$ in $x$ by $\beta$ (resp. $\alpha$).  Clearly $\bar\phi_\gamma$ is a continuous extension of $\phi_\gamma$ to $\Sigma_D$ and it is not injective. If there were a continuous map $\bar\psi_\gamma\colon\Sigma_\gamma\to\Sigma_D$ such that $\bar\psi_\gamma(x)=\psi_\gamma(x)$ for all $x\in K_\gamma$, then $\bar\psi_\gamma\circ\bar\phi_\gamma\colon\Sigma_D\to\Sigma_D$ would be continuous. Since $\psi_\gamma\circ\phi_\gamma(x)=x$ for all $x\in B_\gamma$ and $B_\gamma$ is a dense subset of $\Sigma_D$, $\bar\psi_\gamma\circ\bar\phi_\gamma (x)=x$ would hold for all $x\in \Sigma_D$, a contradiction to the non-injectivity of $\bar\phi_\gamma$. \end{proof}




\subsection*{Acknowledgments} 
This research was supported by the JSPS KAKENHI 23K20220.

      \bibliographystyle{amsplain}

\begin{thebibliography}{10}



 
 
 

  







 
 \bibitem{AU68} A. V. Aho and J. D. Ullman,
  {\it The theory of languages}, Math. Systems Theory {\bf 2} (1968), 97--125.



\bibitem{Bow71} R. Bowen, {\it Entropy for group endmorphisms and homogeneous spaces}, Trans. Amer. Math. Soc. {\bf 153} (1971), 401--414.



\bibitem{DemZei98}
A. Dembo and O. Zeitouni, 
{\it Large deviations techniques and applications}, volume~38 of 
  {\it Applications of Mathematics}.
 Springer-Verlag, New York, second edition, 1998.


\bibitem{EKW94} A. Eizenberg, Y. Kifer and B. Weiss, {\it Large deviations for $\mathbb Z^d$-actions},  Commun. Math. Phys. {\bf 164} (1994), 433--454.

\bibitem{Ell85} R.~S. Ellis, {\it Entropy, large deviations, and statistical mechanics}, volume 271 of Grundlehren der Mathematischen Wissenschaften [Fundamental  Principles of Mathematical Sciences]. Springer-Verlag, New York, 1985.



  \bibitem{HI05} T. Hamachi and K. Inoue, {\it Embedding of shifts of finite type into the Dyck shift}, Monatshefte f\"ur Mathematik {\bf 145} (2005), 107--129.






\bibitem{Kel91}  G. Keller, {\it Circular codes, loop counting, and zeta-functions},  Journal of Combinatorial Theory, Series A, {\bf 56} (1991), 75--83.


\bibitem{Kif90}
Y. Kifer,
{\it Large deviations in dynamical systems and stochastic processes}, 
 Trans. Amer. Math. Soc. {\bf 321} (1990), 505--524.

 \bibitem{Kif94}
 Y. Kifer, {\it Large deviations, averaging and periodic orbits of dynamical systems},  Commun. Math. Phys. {\bf 162} (1994), 33--46.
  
 \bibitem{Kri74} W. Krieger, 
 {\it On the uniqueness of the equilibrium state}, Math. Systems Theory 
 {\bf 8} (1974/75), 97--104. 

 

 \bibitem{Kri00} W. Krieger, {\it On a syntactically defined invariant of symbolic dynamics},  Ergodic Theory Dyn. Syst. {\bf 20} (2000), 501--516.


  
 
 
 




 \bibitem{M04} K. Matsumoto,  {\it A simple purely infinite $C^*$-algebra associated with a lambda-graph system of the Motzkin shift},  Math. Z.  {\bf 248} (2004), 369--394.


 \bibitem{Mey08} 
T. Meyerovitch, 
{\it Tail invariant measures of the Dyck shift},  Israel J. Math. {\bf 163} (2008), 61--83.



\bibitem{OrePel89}
S. Orey and S. Pelikan, {\it Deviations of trajectory averages and the defect in Pesin's formula
  for Anosov diffeomorphisms},
 Trans. Amer. Math. Soc.
{\bf  315} (1989), 741--753.


 





\bibitem{PS05} C.-E. Pfister and W. G. Sullivan, {\it Large deviations estimates for dynamical systems without the specification property. Application to the $\beta$-shifts}, 
Nonlinearity {\bf 18} (2005), 237--261. 


\bibitem{Rue04} D. Ruelle, {\it Thermodynamic formalism. The mathematical structures of classical equilibrium statistical mechanics,} Second edition. Cambridge University Press, 2004.
\bibitem{Sch97} J. Schmeling,
{\it Symbolic dynamics for $\beta$-shifts and self-normal numbers},
Ergodic Theory Dyn. Syst. {\bf 17} (1997), 675--694.

\bibitem{Str25} D. W. 
Stroock, {\it Probability theory, an analytic view. Third edition}, Cambridge University Press, Cambridge, 2025. 







\bibitem{Tak84}
Y. Takahashi, {\it Entropy functional (free energy) for dynamical systems and their
  random perturbations},
 Stochastic analysis ({K}atata/{K}yoto, 1982), 
 437--467, North-Holland Math. Library, {\bf 32}, North-Holland, Amsterdam, 1984. 

 
 \bibitem{T23} H. Takahasi,  {\it Exponential mixing for heterochaos baker maps and the Dyck system}, J. Dyn. Differ. Equ. doi:10.1007/s10884-024-10370-x
 
 




\bibitem{T24} H. Takahasi,  {\it Distributions of periodic points for the heterochaos baker maps and the Dyck shift}, 
arXiv:2409.01261

 \bibitem{W70} 
   B. Weiss, {\it Intrinsically ergodic systems}, Bull. Am. Math. Soc. {\bf 76} (1970), 1266--1269.


\bibitem{You90} L.-S., Young,
{\it Some large deviations results in dynamical systems}, Trans. Amer. Math. Soc.  {\bf 318} (1990), 525--543.




  
  \end{thebibliography}

\end{document}